\documentclass[11pt]{article}
\usepackage{bbm}
\usepackage{geometry}
\usepackage{color}
\usepackage{amsfonts}
\usepackage{latexsym, amssymb, amsmath, amscd, amsthm, amsxtra}
\usepackage{mathtools}
\usepackage{enumerate}
\usepackage[all]{xy}
\usepackage{mathrsfs}
\usepackage{fancyhdr}
\usepackage{listings}
\usepackage{hyperref}

\usepackage{mathrsfs}
\usepackage{hyperref}
\usepackage{mathtools}
\usepackage{setspace}

\newtheorem{thm}{Theorem}[section]

\newtheorem{lem}[thm]{Lemma}

\newtheorem{prop}[thm]{Proposition}

\newtheorem{definition}[thm]{Definition}

\def\ism {\mu_{\epsilon h}^+}     
\def\isn {\mu_{0}^{+}}            

\def\parz{\mathcal{Z}^+_{\phi}(\epsilon h)}    
\def\parf#1{\mathcal{Z}^+_{\phi}(\epsilon {#1})}  
\def \part{\mathcal{Z}^+_{\phi}}

\def \fkh#1{\phi_p^{{#1},\epsilon h}}
\def \vfkh#1{\varphi_p^{{#1},\epsilon h}}
\def \fkf#1#2{\phi_p^{{#1},\epsilon {#2}}}
\def \vfkf#1#2{\varphi_p^{{#1},\epsilon {#2}}}
\def \fk#1{\phi_p^{#1,0}}
\def \vfk#1{\varphi_p^{#1,0}}

\newcommand{\lamn}{\Lambda_N}                     

\newcommand{\eps}{\epsilon}

\def\mbb#1{\mathbb{#1}}
\def\mcc#1{\mathcal{#1}}

\def\pari{\partial_{\mathrm i}}
\def\pare{\partial_{\mathrm e}}

\begin{document}

	\title{Long range order for three-dimensional random field Ising model throughout the entire low temperature regime}
\author{Jian Ding  \\ Peking University \and Yu Liu  \\ Peking University \and Aoteng Xia \\ Peking University}

\maketitle

\begin{abstract}
For $d\geq 3$, we study the Ising model on $\mathbb Z^d$ with random field given by $\{\epsilon h_v: v\in \mathbb Z^d\}$ where $h_v$'s are independent normal variables with mean 0 and variance 1. We show that for any $T < T_c$ (here $T_c$ is the critical temperature without disorder), long range order exists as long as $\epsilon$ is sufficiently small depending on $T$. Our work extends previous results of Imbrie (1985) and Bricmont--Kupiainen (1988) from the very low temperature regime to the entire low temperature regime.
\end{abstract}

\section{Introduction}

For $d\geq 3$, we consider the $d$-dimensional lattice $\mathbb Z^d$ where $u, v\in \mathbb Z^d$ are adjacent (and we write $u\sim v$) if their $\ell_1$-distance is 1.
For $N\geq 1$, let $\Lambda_N=[-N,N]^d \cap \mathbb Z^d$ be the box of side length $2N$ centered at the origin $o$. For $v\in \mathbb Z^d$, let $h_v$'s be independent normal variables with mean 0 and variance 1 (we denote by $\mathbb P$ and $\mathbb E$ the measure and expectation with respect to  $h=\{h_v\}_{v\in \mathbb Z^d}$, respectively). For $\epsilon \geq 0$, we define the random field Ising model (RFIM) Hamiltonian $H^{\pm, \Lambda_N, \epsilon h}$ with the plus (respectively minus) boundary condition and the external field  $\eps h=\{ \epsilon h_v :v \in \mathbb{Z}^d \}$ by
\begin{equation}
\label{def_h}
H^{\pm, \Lambda_N, \epsilon h} (\sigma)= - \left(\sum_{u,v \in \Lambda_N,\, u\sim v} \sigma_u \sigma_v \pm \sum_{ u \in \Lambda_N, v\in \Lambda_N^c,\, u\sim v} \sigma_u + \sum_{u \in \Lambda_N} \epsilon h_u \sigma_u\right),
\end{equation}
where $\sigma \in \{-1, 1\}^{\Lambda_N}$.
For $T>0$, we define $\mu^{\pm}_{T, \Lambda_N, \epsilon h} $ to be the Gibbs measure on $\{-1,1\}^{\Lambda_N}$ at temperature $T$ by
\begin{equation}
\label{def_mu}
\mu^{\pm}_{T, \Lambda_N, \epsilon h} (\sigma)= \frac{1}{\mathcal{Z}^\pm_{T, \Lambda_N,\mu} (\epsilon h)} e^{-\frac{1}{T} H^{\pm, \Lambda_N, \epsilon h} (\sigma)},
\end{equation}
where (again) $\sigma\in \{-1, 1\}^{\Lambda_N}$ and $\mathcal{Z}^\pm_{T, \Lambda_N,\mu}( \epsilon h)$ is the partition function given by
\begin{equation}
\label{def_z}
\mathcal{Z}^\pm_{T, \Lambda_N,\mu}( \epsilon h) = \sum_{\sigma \in \{-1,1\}^{\Lambda_N}} e^{-\frac{1}{T} H^{\pm, \Lambda_N, \epsilon h} (\sigma)}.
\end{equation}
(Note that $\mu^\pm_{T, \Lambda_N, \epsilon h}$ and $\mathcal Z^\pm_{T, \Lambda_N,\mu} (\epsilon h)$ are random variables depending on  the external field $\eps h $.) Similarly, we denote the measure without disorder by $\mu^{\pm}_{T, \Lambda_N,0}$. Throughout the paper, we let $T_c = T_c(d)$ be the critical temperature for the Ising model without disorder and we consider $T<T_c$. A key quantity of interest in this paper is the boundary influence, defined as
$$
m_{T,\Lambda_N, \epsilon h} = \mu^{+}_{T, \Lambda_N, \epsilon h}(\sigma_o = 1) - \mu^-_{T, \Lambda_N, \epsilon h}(\sigma_o = 1)\,.
$$
Our main theorem concerns the existence of long range order, that is, the boundary influence $m_{T,\Lambda_N, \epsilon h}$ stays above a positive constant as $N\to \infty$ (for a typical instance of the disorder).
\begin{thm}\label{thm-main}
For any $d\geq 3,$  $ T< T_c$ and any $\gamma>0$, there exists $\epsilon_ 0 = \epsilon_0(d, T, \gamma)$ such that for all $\epsilon \leq \epsilon_0$ we have
$$\mbb P(m_{T,\Lambda_N, \epsilon h} \geq \epsilon_0) \geq 1 - \gamma \mbox{ for all } N \geq 1\,.$$
\end{thm}
There is a long history for the study of long range order for the random field Ising model.
With strong disorder, i.e., when $\epsilon$ is large, this is relatively an easy question. It was shown in \cite{Ber85, FI84, CJN18} (see also \cite[Appendix A]{AP19}) that for any $d\geq 2$, the boundary influence decays exponentially in $N$ (so in particular there exists no long range order) at any temperature as long as $\epsilon = \epsilon (d)$ is large enough. On the contrary, the question becomes substantially more challenging with weak disorder, i.e., when $\epsilon$ is small. There is a \emph{vast} literature on the random field Ising model in physics including both numeric studies and non-rigorous derivations. Among these, the most influential physics prediction is arguably due to \cite{IM75}, which predicts that long range order exists at low temperatures for $d\geq 3$ but not for $d=2$. It would not be possible to provide a complete list on relevant literature in physics, and a partial list on the random field Ising model for $d\geq 3$ with some emphasis on ``transition in temperature'' (and thus most relevant to Theorem~\ref{thm-main}) includes \cite{GM83, YN85, PF85, HFBT93, BFHHRT95}. While we feel that Theorem~\ref{thm-main} should have been speculated by people, we are unable to locate an explicit and precise reference.

Somewhat interestingly, for the random field Ising model it seems that it has happened for multiple times that debates among physicists were solved/clarified with the help of rigorous work from mathematicians. For instance, there has been controversy over the prediction of \cite{IM75} for quite some time, and it was finally proved to be correct by \cite{Imbrie85, BK88} (see also \cite[Chapter 7]{bovier06} and \cite{DZ21}) for $d\geq 3$ and by \cite{AW90} for $d=2$ (see also \cite{Chatterjee18, AP19, DX21, AHP20}). It is also worth noting that for $d=2$ we now have a fairly good quantitative understanding including exponential decay \cite{DX21, AHP20} and the correct scaling for (a notion of) the correlation length \cite{DW20}, over both of which physicists seemed to have debates for a long time.

It seems that until very recently, in the study (at least mathematical study) of the random field Ising model,  the high temperature regime refers to the case when $T$ is very large and the low temperature regime refers to the case when $T$ is sufficiently small. In the direction of proving correlation decay,  a  recent new correlation inequality for the Ising model \cite{DSS22}, combined with classic results for exponential decay of the Ising model without disorder at $T>T_c$ \cite{ABF87} (see also \cite{DTV18, DRT19}), implies exponential decay for the RFIM as long as $T > T_c$. In the direction of proving long range order, the method in the classic papers \cite{Imbrie85, BK88} is based on a sophisticated scheme of renormalization group theory, and as a result it seems very difficult (if possible at all) to extend to the entire low temperature regime. The starting point of our proof is the extension of the Peierls argument in \cite{DZ21}, but it is fair to say that one has to conquer major obstacles in order to prove long range order for $T<T_c$ even assuming familiarity of \cite{DZ21}. Extending \cite{DZ21}, our proof employs a Peierls type of \emph{equality} as opposed to an \emph{inequality} when Peierls type of argument (initiated in \cite{Peierls1936}) was employed in all previous works to our best knowledge.
As a result, our proof
 combines many ingredients such as the greedy lattice animal \cite{Chalker83, FFS84} (which already appeared in \cite{DZ21}), the connection to FK-Ising model \cite{FK72, ES88} (see also \cite{Grimmett06}), the coarse graining method \cite{Pis96, Bodineau05}, as well as the spatial mixing property for the FK-Ising model \cite{DCGR20, Alexander98} and a random partitioning method from \cite{CKR05} (which allows us to apply the spatial mixing property in an effective manner). We feel that our proof approach is non-trivial and at least somewhat novel. As a result, we present a high-level overview in Section~\ref{sec:outline-and-prelim} emphasizing the importance of the FK-Ising model and the coarse graining method in our proof outline, and then we present a detailed discussion on our proof outline at the beginning of Section~\ref{sec:proof-sketch}.

\section{Proof overview and preliminaries} \label{sec:outline-and-prelim}

In this section, we present an overview discussion for our proof strategy, from which we see that the FK-Ising model (as reviewed in Section~\ref{sec:coupling}) and the coarse graining method (as reviewed in Section~\ref{sec-coarse-graining-framework}) are of fundamental importance to our proof.

\subsection{Overview of our proof}\label{sec:proof-outline}

In what follows we may drop the subscripts $T$ and region $\lamn$ in the context of no ambiguity. For example, $\mu^{\pm}_{T, \Lambda_N, \epsilon h}$ will be abbreviated to $\mu^{\pm}_{\epsilon h}$ and similarly for $\mu^{\pm}_{T, \Lambda_N, 0}$ and $\mu^{\pm}_0$. When we wish to emphasize the underlying region for the Ising/FK-Ising model, we simply add it to the subscript. For instance, we denote by $\phi^{\xi,0}_{p,U}$ the FK-Ising model on $U$ with boundary condition $\xi$ and no external field.
Our goal is to show that for any fixed $T<T_c$, long range order exists for sufficiently small $\epsilon$ which may depend on $T$, as in Theorem~\ref{thm-main}. It suffices to prove the following theorem. 
	\begin{thm}\label{thm-main-reformulation}
        For any $d\geq 3$ and $T< T_c$ and any $\gamma>0$, there exists $\epsilon_ 0 = \epsilon_0(d, T, \gamma)$ such that for all $\epsilon \leq \epsilon_0$ we have
        $$\mbb P(\ism
        (\sigma_o=-1) \leq 1/2 - \epsilon_0) \geq 1 - \gamma \mbox{ for all } N \geq 1\,.$$
    \end{thm}
    For convenience, we will write the proof for $d = 3$ and the adaption to $d > 3$ is verbatim. 
    In order to prove Theorem~\ref{thm-main-reformulation}, the very basic idea of our proof lies in the classic Peierls argument. The first challenge in applying the Peierls argument to the random field Ising model is that the disorder has non-negligible influence on the change of Hamiltonian. This challenge was addressed in \cite{DZ21} where the Peierls mapping was extended to the joint space of Ising configurations and the random external field; in particular, the signs for disorder in the simply connected set enclosed by the outmost boundary for the sign cluster of the origin were flipped together with the Ising spins therein, while the change of partition function from the flipping procedure is $e^{\delta_\epsilon L}$ with $\delta_\epsilon \to_{\epsilon \rightarrow 0 }0$ (here $L$ is the size for the outmost boundary). 

In \cite{DZ21}, $T$ is sufficiently small so that the gain from flipping the Ising spins is very large and easily absorbs multiplicity of the Peierls mapping. In the contrast, in this work $T$ can be just slightly below $T_c$, and as a result a much more refined analysis is required.
    It seems the only plausible route here is to use $\isn(\sigma_o = 1) > 1/2$ for $T < T_c$ (which is essentially the definition of $T_c$) and to make comparisons between what happens along the Peierls mapping for the Ising model without and with external field. This is incorporated in the following proposition.
\begin{prop}\label{prop:main-thm-restatement}
		For any $ T < T_c$ and any $\theta_1,\theta_2, \gamma>0$, there exists  $\eps_0=\eps_0(\gamma, \theta_1,\theta_2)>0$ such that for any $\eps<\eps_0$
		we have that with $\mbb P$-probability at least $1 -\gamma$
		$$\ism (\sigma_o=-1)\leq e^{\theta_1}\isn(\sigma_o=-1)+\theta_2\,.$$
	\end{prop}
 
\begin{proof}[Proof of Theorem~\ref{thm-main-reformulation} assuming Proposition~\ref{prop:main-thm-restatement}]	Since $\isn(\sigma_o=-1) < 1/2$ for $T<T_c$, we can take $\theta_1, \theta_2$ small enough so that $e^{\theta_1}\isn(\sigma=-1)+\theta_2 < \frac{1}{2} - \delta$ for some  $\delta = \delta(T)>0$. Combined with Proposition~\ref{prop:main-thm-restatement}, this proves Theorem~\ref{thm-main-reformulation}.
 \end{proof}
We wish to prove Proposition~\ref{prop:main-thm-restatement} by a Peierls type of argument. With a moment of thinking, one can be convinced that in order for the Peierls argument to succeed in a plausible way, it is \emph{essentially} necessary that the ``cluster'' flipped in the Peierls mapping exhibits an exponential decay in its surface area. This requirement indeed holds for previous applications of the Peierls argument since (as far as we know) previously the Peierls argument was applied at sufficiently low temperatures where the minus spin cluster under the plus boundary condition does exhibit the required exponential decay (it is also important that the rate in the exponential decay is large for small $T$).  However, this may not be the case when $T$ is slightly below $T_c$: it was conjectured that there exists a near-supercritical regime where  minus spins percolate under the plus boundary-condition \cite{BFL82}. Assuming this conjecture, we will have no decay for the size of the minus cluster, let alone an exponential decay.
 
In light of this, a natural alternative is to work with the FK-Ising cluster which does have a decaying probability for the FK-cluster of the origin to be large but disconnected from the boundary \cite{Pis96, Bodineau05}. For our analysis later, it would be ideal that this decay rate is exponential in the boundary size of the cluster (at least in the size for the outmost boundary). Unfortunately, this is not the case when $T$ is near $T_c$: with probability $e^{-O(N^2)}$ all edges on the boundary of $\Lambda_{N/2}$ are closed, and given this event with non-vanishing probability the cluster of the origin has volume of order $N^3$ and also has boundary size of order $N^3$ (it is expected that when $T$ is near $T_c$ the size of the outmost boundary should also be of order $N^3$). In order to address this, we apply the coarse graining method where we consider the cluster in a lower resolution: roughly speaking, we consider a box of size $\mathsf q$ as occupied by the cluster as long as it intersects with the cluster. We will review coarse graining method in Section~\ref{sec-coarse-graining-framework}, which is of fundamental importance to our analysis. 

Even with the aforementioned coarse graining, it is plausible that the rate for the exponential decay approaches 0 when  $T \uparrow T_c$, and as a result still there is very little room to spare in our comparison of measures for events concerning a component with boundary size $L$: we need to show that the ratio between the measures with and without the external field is within an $e^{\delta L}$ multiplicative factor (where $\delta$ is sufficiently small depending on $T$). It is then perhaps not surprising that such high precision comparison incurs substantial challenge. The main underlying intuition is that, the main influence from the disorder in the aforementioned comparison comes from a thin shell around the outmost boundary (and thus has volume of order $L$). In order to make this intuition rigorous, we have to slightly modify the definition of the cluster when constructing the Peierls mapping so that this thin shell is surrounded by good boxes of size $\mathsf q$ which will ``shield'' the influence of the disorder (see Section~\ref{sec-coarse-graining-framework} for the definition of a `good' box).  The proof is incorporated in Section~\ref{sec:proof-sketch} where we set up the framework and present the flow of our proof. Two main ingredients, including the spatial mixing and the actual comparison of the measures restricted in the aforementioned thin shell, are postponed until later sections.

In the subsequent two subsections, we review the FK-Ising model (and the coupling between the FK-Ising model and the Ising model) and we also review the framework of the coarse graining method in the context of the FK-Ising model.

\subsection{FK-Ising model and Edward-Sokal coupling}\label{sec:coupling}

In this subsection, we review Edward-Sokal coupling between the Ising model and the FK-Ising model \cite{ES88} (see also \cite{Grimmett06} for an excellent account on the topic of random cluster models, where the FK-Ising is a particular and important example). Although our interests are for lattices, we describe this coupling in more general setting partly since we may need to apply this for general domains in lattices. For a finite graph $G=(V,E)$, let $V$ be the disjoint union of $W$, $U_+$ and $U_-$, and let  $\{h_x:x\in V\}$ be an external field. We define $E(V)$ to be the collection of edges within $V$ and define $E(U, V)$ to be the collection of edges between $U$ and $V$.

Then, we can naturally extend \eqref{def_h} and define the Hamiltonian (with respect to the plus boundary condition on $U^+$, the minus boundary condition on $U^-$ and the external field $h$) for an Ising spin configuration $\sigma\in\Sigma=\{-1,1\}^V$ by
$$ H^{U, \pm}_{G, h}(\sigma)= - \big(\sum_{\{x,y\}\in E(W)}\sigma_x\sigma_y+\sum_{\{x,y\}\in E(W, U_+)}\sigma_x-\sum_{\{x,y\}\in E( W, U_-)}\sigma_x+\sum_{x\in V}h_x \sigma_x\big)\,.$$
Then the Ising measure $ \mu^{U,\pm}_{T, G, h}$ is a probability measure on $\Sigma$ such that $ \mu^{U,\pm}_{T, G, h}(\sigma) \propto e^{-\frac{1}{T} H^{U, \pm}_{G, h}(\sigma)}$
where $\propto$ means ``proportional to''.

Similarly, on the graph $G=(V,E)$, we can consider a percolation configuration $\omega\in\Omega=\{0,1\}^E$ where $0$ indicates that an edge is \emph{closed} and $1$ indicates an edge is \emph{open}. For $\omega \in \Omega$ and $x \neq y\in V$, we say $x$ is \emph{connected} to $y$ in $\omega$ if $x$ is connected to $y$ by an open path in $\omega$. Let $\eta(\omega)$ be the number of open clusters in $\omega$.
Then the FK-Ising model on $G$ with parameter $p \equiv 1-\exp(-\frac{2}{T})$ is a probability measure on  percolation configuration $\Omega$ given by
\begin{equation}\label{eq-def-random-cluster}
\phi_{p,G}(\omega)= \frac{1}{\mathcal Z_\phi}\prod_{e\in E}p^{\omega(e)}(1-p)^{1-\omega(e)}\times 2^{\eta(\omega)} \mbox{ for } \omega \in \Omega\,,
\end{equation}
where $\mathcal Z_\phi$ is a normalizing constant for $\phi_{p,G}$. 
We can further pose certain boundary condition on a subset $U$ of $V$. For instance, the \emph{wired} and \emph{free} boundary condition on $U$ can be viewed as conditioning on the edges in $E(U)$ being open and closed respectively.

The Edward-Sokal coupling \cite{ES88} between the FK-Ising model and the Ising model with $p=1-\exp(-\frac{2}{T})$ can be described as follows:
for $G=(V,E),$ consider a probability measure $\pi$ on $\Sigma\times \Omega$ defined by $$\pi(\sigma,\omega)=\frac{1}{\mathcal Z_\pi}\prod_{e=\{x,y\}\in E}[(1-p)\delta(\omega(e),0)+p\delta(\omega(e),1)\delta(\sigma_x,\sigma_y)],$$
where $\delta(\cdot,\cdot)$ is a  $\{0,1\}$-valued function such that $\delta(a,b) =1$ if and only if $a=b$ and  $\mathcal Z_{\pi}$  is the normalizing constant. Then we have the following properties (see, e.g., \cite{Grimmett06}):
\begin{enumerate}[(i)]
   \item Marginal on $\Sigma$: for a fixed $\sigma\in \Sigma$, summing over all $\omega\in\Omega$ we have 
   $$\pi(\sigma)=\sum_{\omega\in\Omega}\pi(\sigma,\omega)=\frac{1}{\mathcal Z_\pi}\prod_{e=\{x,y\}\in E}[(1-p)+p\delta(\sigma_x,\sigma_y)]\propto (1-p)^{\sum\limits_{\{x,y\}\in E}\mathbbm{1}_{\{\sigma_x\neq \sigma_y\}}}\,.$$
   This coincides with the Ising measure $\mu^{U,\pm}_{T, G, h}$ with $h = 0$, $U^\pm = \emptyset$ and  $1-p=\exp(-\frac{2}{T})$.
   
   \item Marginal on $\Omega$: for a fixed $\omega\in \Omega$, summing over all   $\sigma\in \Sigma$ gives
   $$\pi(\omega)=\sum_{\sigma\in \Sigma}\pi(\sigma,\omega)=\frac{1}{\mathcal Z_\pi}\prod_{\omega(e)=0}(1-p)\sum_{\sigma\in\Sigma}\prod_{\omega(e)=1}p\delta(\sigma_x,\sigma_y).$$ 
   In order for the preceding product to be non-zero, the two spins on the endpoints of every open edge must agree and thus the spin on every open cluster is the same. As a result, there are $2^{\eta(\omega)}$ different choices for $\sigma$, thus we have
   $$\pi(\omega)=
   \frac{1}{\mathcal Z_{\pi}}\prod_{e\in E}[p^{\omega(e)}(1-p)^{1-\omega(e)}]\times 2^{\eta(\omega)}\,,$$
   which corresponds to the measure of the FK-Ising model on $\{0,1\}^E.$
   
   \item Conditioned on spin configuration: for a fixed $\sigma\in\Sigma$,
   let $E_1=\{\{x,y\}\in E: \sigma_x=\sigma_y\}$ and $E_2=E\setminus E_1$. Dividing $\pi(\sigma, \omega)$ by $\pi(\sigma)$ yields that
   $$\pi|_{\sigma}(\omega)\propto \prod_{e\in E_2}\delta(\omega(e),0)\prod_{e\in E_1}[(1-p)\delta(\omega(e),0)+ p \delta(\omega(e),1)].$$
   Thus, under $\pi|_{\sigma}$,  the edges in $E_2$ must be closed and in addition edges in $E_1$ are open with probability $p$ and closed with probability $1-p$ independently. In other words, the conditional measure for $\omega$ can be viewed as a Bernoulli bond percolation within each spin cluster.
   \item Conditioned on percolation configuration: For a fixed $\omega\in\Omega$, let $E_3 =\{e\in E:\omega(e)=1\}$ and $E_4 =E\setminus E_3$. Then we have 
   $$\pi|_{\omega}(\sigma) \propto\prod_{e=\{x,y\}\in E_3}p\delta(\sigma_x,\sigma_y).$$
   Thus, the conditional
   measure $\pi|_{\omega}$ is uniform among spin configurations where each open cluster in $\omega$ receives the same spin.
\end{enumerate}
Furthermore, the above coupling can be extended to the case with external field $\{h_x:x\in V\}$. To this end, we define $$\pi_h(\sigma,\omega)=\frac{1}{\mathcal Z_{\pi_h}}\prod_{e=\{x,y\}\in E}[(1-p)\delta(\omega(e),0)+p\delta(\omega(e),1)\delta(\sigma_x,\sigma_y)]\exp(\frac{1}{T}\sum_{x\in V}\sigma_xh_x)\,,$$
where $\mathcal Z_{\pi_h}$ is the normalizing constant. 
By a straightforward computation, we see 
\begin{equation}\label{eq-partition-function-equal}
\mathcal Z_{\pi_h} = \sum_{\sigma\in \Sigma}\sum_{\omega\in \Omega}\pi_h(\sigma,\omega)= \sum_{\sigma\in \Sigma}\exp[\frac{1}{T}(\sum_{\{x, y\}\in E} \sigma_x \sigma_y + \sum_{x\in V} \sigma_x h_x)] \cdot \exp(-\frac{1}{T}|E|)\,,
\end{equation}
implying that $\mathcal Z_{\pi_h}$ and the partition function $\mathcal Z_{\mu_h}$ for the Ising measure $\mu_{T,V,h}$  are up to a factor of $\exp(-\frac{1}{T}|E|)$, which in particular does not depend on the external field $h$.
One can verify that the aforementioned properties (in the case without external field) can be extended as follows: 
\begin{enumerate}[(i)]
    \item For a fixed $\sigma\in \Sigma$, the additional term $\exp(\frac{1}{T}\sum\limits_{x\in V}\sigma_x h_x)$ is a constant in $\omega$. As a result, the marginal distribution of $\pi_h$ on $\Sigma$ is exactly the Ising measure with external field $h$. In addition, the conditional distribution for percolation configuration given a fixed $\sigma$ is also equivalent to a Bernoulli bond percolation within each spin cluster (i.e., the same as the case for $h=0$).
    \item For a fixed $\omega \in \Omega$, in order for $\pi_h(\sigma, \omega)$ to be non-zero, each open cluster must have the same spin. Thus, writing $h_A = \sum_{x\in A} h_x$ for any subset $A \subset \mathbb Z^3$, we have
    $$\sum_{\sigma\in \Sigma}\pi_h(\sigma,\omega)=\frac{1}{\mathcal Z_{\pi_h}} \prod_{e\in E}p^{\omega(e)}(1-p)^{1-\omega(e)}\prod_{j=1}^{\eta(\omega)}2 \cosh(h_{\mcc C_j}/T),$$
    where $\mcc C_1,\mcc C_2,\cdots,\mcc C_{\eta(\omega)}$ are all the  (random) open clusters in $\omega$ and $\cosh z = \frac{e^z + e^{-z}}{2}$. Therefore, conditioned on a fixed $\omega\in \Omega$, each open cluster $\mcc C_j$ must have the same spin and this spin is plus with probability
    $\frac{\exp(h_{\mcc C_j}/T)}{2\cosh( h_{\mcc C_j}/T)}$.
\end{enumerate}

In this paper, we mainly consider the Ising measure on the box $\lamn$ as defined in \eqref{def_mu}.
We define the  internal and external boundary of a set $A$ respectively by
\begin{align*}
\pari A&=\{x\in A:\text{ there exists }y\in A^c \text{ with } x\sim y\}\,,\\
\pare A&=\{y\in A^c:\text{ there exists }x\in A \text{ with } x\sim y\}\,.
\end{align*}
We define the edge boundary of a set $A$ by $\partial^{\mathrm{e}}A = \{(x, y): x\sim y, x\in A, y\in A^c\}$. 
Then the corresponding FK-Ising measure of $\mu^{+}_{T, \Lambda_N, \epsilon h}$ is constructed by identifying the  external boundary $\pare \lamn$ to be one point and by conditioning on the event that the open cluster connected to it has the plus spin. For notation convenience, denote by $\mcc C_*$  the open cluster connected to  $\pare\lamn$, and denote by $\mcc C_0$  the open cluster containing the origin $o$. It is possible that $\mcc C_0=\mcc C_*$, i.e., the origin is connected to the boundary $\pare \lamn$ in $\omega$. In light of this, we let $s=1$ if $\mcc C_0=\mcc C_*$ and let $s=0$ otherwise.
In addition, let $\mcc C_1,\mcc C_2,\cdots,\mcc C_{\eta(\omega)-2+s}$ be the rest of the open clusters.
 Then we can write the corresponding FK-Ising measure $\phi^{+, \epsilon h}_{p,T, \lamn}$ as 
\begin{equation}\label{eq-def-FK-external-field}
\phi^{+, \epsilon h}_{p,T, \lamn}(\omega)=\frac{1}{\parz } \prod_{e\in E(\Lambda_N)\cup \partial^{\mathrm e} \Lambda_N}\hspace{-1.5em}p^{\omega(e)}(1-p)^{1-\omega(e)}\prod_{j=s}^{\eta(\omega)-2+s}\hspace{-1em}2 \cosh(\eps h_{\mcc C_j}/T)\times \exp(\eps h_{\mcc C_*}/T)\,,
\end{equation}
where $\parz$ is the partition function of $\phi^{+, \epsilon h}_{p,T, \lamn}$.  Note that on the right hand side, the weight of $\mcc C_*$ is an exponential term instead of twice the hyperbolic cosine for the reason that the spin on the  external boundary is fixed to be plus. In addition, we denote the FK-Ising measure with the plus boundary condition and without disorder by $\phi^{+, 0}_{p,T, \lamn}$ (that is, $\phi^{+, \epsilon h}_{p,T, \lamn}$ with $\epsilon = 0$), and we denote its partition function by $\mathcal Z^+_{\phi}$. 

Recall that for brevity of notation, in what follows we will omit the temperature $T$ and region $\lamn$ in the subscripts of the Ising and FK-Ising measures in the context of no ambiguity. With this convention we have
\begin{equation}\label{eq:expressionstep1}
\ism(\sigma_0=-1)=\sum_{C_0\subset \lamn}\fkh +(\mcc C_0=C_0)\frac{\exp(-\eps h_{C_0}/T)}{2\cosh(\eps h_{C_0}/T)}\,,
\end{equation}
where the sum is over all connected subset $C_0$ that contains the origin and is not connected to $\pare\lamn$ (the latter is implied by $C_0 \subset \Lambda_N$).

\subsection{Framework of coarse graining}\label{sec-coarse-graining-framework}
	
	In this subsection, we describe the method of coarse graining in the context of FK-Ising model, which plays an important role in various places of our proof and in particular in the definition for the coarsening of the FK-cluster used in our Peierls mapping. 
	The method of coarse graining for the FK-Ising model was first introduced in \cite{Pis96} for $T < \hat T_c$ where $\hat T_c$ is the so-called slab percolation critical temperature. Later in \cite{Bodineau05} it was shown that $\hat T_c = T_c$ for the FK-Ising model.	In this paper our presentation follows the framework of \cite{DCGR20} more closely, where coarse graining was applied to show the spatial mixing property for the FK-Ising model and the exponential decay of truncated correlation for the Ising model.

	As a notation convention, we typically use the mathsf font to denote objects related to the coarse graining. For an integer $\mathsf q\ge 1$ and for $\mathsf v=(\mathsf v^1, \mathsf v^2, \mathsf v^3)\in \mbb Z^3$, write $\mathsf q \mathsf v = (\mathsf q \mathsf v^1, \mathsf q \mathsf v^2, \mathsf q \mathsf v^3)$ and  write  $\mathsf q\mbb Z^3 = \{\mathsf q \mathsf v: \mathsf v\in \mbb Z^3\}$. 
	For a vertex $\mathsf v \in \mbb Z^3$,  let $\mathsf Q_{\mathsf v}$ be the box of side length $2\mathsf q$ centered at $\mathsf q\mathsf v$, i.e., 
	\begin{equation}\label{eq-def-mathsf-Q}
	\mathsf Q_{\mathsf v}=\mathsf q \mathsf v+[- \mathsf q, \mathsf q]^3\cap \mbb Z^3\,.
	\end{equation}
	For  $\mathsf V\subset \mbb Z^3$, we define
	$$\mcc Q_{\mathsf V}=\{\mathsf Q_{\mathsf v}:  \mathsf v\in \mathsf V\} \mbox{ and } \mathsf Q_{\mathsf V} = \bigcup_{\mathsf v\in \mathsf V} \mathsf Q_{\mathsf v}\,.$$ 
For later convenience and without loss of generality, we can assume that  $\mathsf q$ divides  $(N+1)$. Define $\mathsf N = \frac{N+1}{q}-1$. We see that $\mathsf{Q_{\Lambda_N}}=\Lambda_{N+1}$. 	For an FK-Ising configuration $\omega$, (following \cite{Pis96}) we say a box $\mathsf Q\in \mcc Q_{\Lambda_{\mathsf N}}$ is \emph{good} in $\omega$ if 
	\begin{enumerate}
		\item $\omega|_{\mathsf Q}$ contains a cluster touching all the 6 faces of $\mathsf Q$, where $\omega|_{\mathsf Q}$ is the restriction of $\omega$ on $\mathsf Q$ (or equivalently, the subgraph on $\mathsf Q$ induced by open edges in $\omega$);
		\item Any open path of diameter at least $\mathsf q$ in $\mathsf Q$ is included in this cluster.
	\end{enumerate}
	From the definition, for a good box $\mathsf Q$, there is a unique cluster in $\mathsf Q$ that touches all 6 faces and thus we can refer to it as \emph{the} cluster in $\mathsf Q$. By \cite{Pis96, Bodineau05} we get that for every $p>p_c$ (recall that $p = 1 - \exp(2/T)$), there exists a constant $\mathsf c_{\mathrm g}>0$ such that for every $\mathsf q>0$ and for any boundary condition $\xi$, 
	\begin{equation}
		\phi_{p, \Lambda_{2\mathsf q}}^\xi[\Lambda_{\mathsf q} \text{ is good }]\geq 1-e^{- \mathsf c_{\mathrm g} \mathsf q}.\label{goodboxprobability}
	\end{equation}
This implies that with probability close to 1 we have that $\mathsf{Q_v}$ is good, regardless of the percolation configuration on edges outside a concentric box with side length $4\mathsf q$. Furthermore, we note that this holds also for $\mathsf v\in\pari\Lambda_{\mathsf N}$ (for which the wired boundary intersects  $\mathsf{Q_v}$).
	
	As mentioned earlier, \eqref{goodboxprobability} was proved in \cite{Pis96} for $p > \hat p_c$ and it was then shown in \cite{Bodineau05} that $\hat p_c = p_c$. In the coarsening, we will consider vertices on $\Lambda_{\mathsf N}$ and we say a vertex $\mathsf v \in \Lambda_{\mathsf N}$ is good if $\mathsf Q_{\mathsf v}$ is good; otherwise, we say $\mathsf v$ is bad.
 By the definition of good box, if $\mathsf v_1, \ldots, \mathsf v_n$ is a neighboring sequence of good vertices, then there is a unique cluster within $ \mathsf Q_{\{\mathsf v_1, \ldots, \mathsf v_n\}}$ that touches every face of every box in $\mcc Q_{\{\mathsf v_1, \ldots, \mathsf v_n\}}$, and we will refer this as the \emph{mesh cluster} in  $\mathsf Q_{\{\mathsf v_1, \ldots, \mathsf v_n\}}$.

	By \eqref{goodboxprobability}, the domain Markov property and \cite[Theorem 1.5]{LSS97}, we see that $\rho$ dominates a Bernoulli percolation $\hat\rho$ with density $\mathsf p_{\mathsf q}$ such that 
	\begin{equation}\label{eq-domination}
	\mathsf p_{\mathsf q}\rightarrow_{\mathsf q\to \infty} 1\,.
	\end{equation}

\section{Proof of Proposition~\ref{prop:main-thm-restatement}}\label{sec:proof-sketch}

In this section we prove Proposition~\ref{prop:main-thm-restatement}. We first set up the framework in Section~\ref{sec:archipelago} by defining the outmost ``blue'' boundary, which has a number of novel complications: 
\begin{itemize}
\item This is defined via a coarse graining manner in the sense that we consider clusters on the coarse grained box $\Lambda_{\mathsf N}$. Essentially we wish to work with the outmost boundary of bad clusters.
\item In order for the analysis of ``localizing disorder'' later, we need this boundary of bad vertices to be surrounded by some layers of good vertices, and as a result we need to further explore from the outmost boundary (see discussions surrounding \eqref{eq-B-prime-blue}).
\item In order to control the influence of conditioning on bad vertices, we introduce some auxiliary random variables and this is also why we use blue vertices instead of bad vertices in our actual construction. 
\end{itemize}
Then, in Section~\ref{sec:sketch-of-proof} we prove Proposition~\ref{prop:main-thm-restatement} assuming three propositions:
\begin{itemize}
\item  Proposition~\ref{prop-compare-small-L} compares the measures with and without disorder when the aforementioned outmost boundary does not exist (this is the case where the FK-cluster of the origin is small). Somewhat interestingly, this is the conceptually more important case (in the sense that most of the probability for the origin to be disconnected from $\pare \Lambda_N$ is in this case) but technically it is much simpler.  
\item Proposition~\ref{prop:coarse-graining} offers the promised exponential decay in the size of the boundary which gives us necessary room to spare in our proof. This is a direct consequence of our definition of good vertices employed in coarse graining, as proved in Section~\ref{sec:coarse-graining-proof}.
\item Proposition~\ref{prop-large-surface-estimation} makes the comparison when the boundary size is large and this comparison loses an exponential factor, which can be absorbed thanks to Proposition~\ref{prop:coarse-graining}.
\end{itemize}    
In Section~\ref{subsec-influence-disorder}, we first prove Proposition~\ref{prop-compare-small-L}. Despite being simpler than Proposition~\ref{prop-large-surface-estimation}, the proof of Proposition~\ref{prop-compare-small-L} is not at all trivial and in fact captures quite some of the conceptual ideas: the simplification comes from the ``fact'' that the region where disorder ``matters'' is of constant size and thus many estimates can be loose (see \eqref{eq-small-to-show-2}). However, the justification for \eqref{eq-small-to-show-2} is not trivial, and in fact for one of the necessary ingredients \eqref{eq:small-comparison-2} we only explain the underlying intuition and refer the formal and complete proof to Section~\ref{sec:localization-disorder} where we prove this in a substantially more difficult scenario (we choose not to provide the formal proof for \eqref{eq:small-comparison-2} since it does not seem to give the reader additional warm up for understanding the harder proof in  Section~\ref{sec:localization-disorder} besides what is offered by the heuristics). Much of the work in Section~\ref{subsec-influence-disorder} is to present the proof of Proposition~\ref{prop-large-surface-estimation}, which includes the following four ingredients.
\begin{itemize}
\item Lemma~\ref{lem:flip-external-field-estimation} combined with Lemma~\ref{lem:comparez} compares the measures before and after flipping the disorder: Lemma~\ref{lem:comparez} controls the change of the partition function when flipping the disorder, which essentially follows from \cite{Chalker83, FFS84, DZ21} and is proved in Section~\ref{sec:coarse-graining-proof}; Lemma~\ref{lem:flip-external-field-estimation} controls the change of the Gibbs weight after flipping the disorder and its proof (as presented in Section~\ref{sec:proof-lem-flip-external-field-estimation}) shares the similar spirit as that for Proposition~\ref{prop:inside-M-event-h-and-no-h}.
\item Proposition~\ref{prop:inside-M-event-h-and-no-h} compares the measures with the same boundary condition on the thin shell around our outmost blue boundary, but one is with disorder and the other is without disorder. As proved in Section~\ref{sec:proof-prop-inside-M-event-h-and-no-h}, the influence of the disorder only comes from this thin shell. 
\item Proposition~\ref{prop:inside-M-event-different-theta} compares the measures without disorder but with different boundary conditions on the shell around our outmost blue boundary. This is proved in Section~\ref{sec:spatial-mixing} with crucial input from \cite{DCGR20,Alexander98}. This is necessary since the boundary conditions on the thin shell without and with disorder can potentially be quite different.
\item Proposition~\ref{prop:0-field-case-connecting-and-S2-red} provides a useful input about our blue/red percolation. The challenge here, as shown in Section~\ref{sec:blue-red-percolation}, is to deal with the conditioning, and this is why we will introduce auxiliary random variables. 
\end{itemize}

\subsection{Outmost blue boundary}\label{sec:archipelago}

In previous applications of the Peierls argument, the change of the configuration typically occurs in a certain simply connected subset; for instance in many cases it was chosen as the simply connected subset enclosed by the sign cluster of the origin.   For our proof, we need to employ a more involved  structure via the method of coarse graining. For $\mathsf k \geq 1$, we say two vertices are $\mathsf k$-neighboring if their $\ell_{\infty}$-distance is at most $\mathsf k$. And thus, we say a set is $\mathsf k$-connected if it is connected with respect to the $\mathsf k$-neighboring relation.   For each set $\mathsf A$, let $\psi(\mathsf A)$ be the collection of vertices enclosed by $\mathsf A$ (formally, a vertex $\mathsf v$ is enclosed by $\mathsf A$ if every path from $\mathsf v$ to $\pari\Lambda_{\mathsf N}$ intersects $\mathsf A$). In addition, we write $\mathrm{Ball}(\mathsf A; \mathsf k) = \{\mathsf v \in \Lambda_{\mathsf N}: d_{\infty}(\mathsf v, \mathsf A) \leq \mathsf k\}$ for $\mathsf k\geq 1$. Furthermore, for convenience of analysis later, we introduce auxiliary random variables $\mathrm{Aux} =  \{\mathrm {Aux}_{\mathsf v}: \mathsf v\in \Lambda_{\mathsf N}\}$ where $\mathrm {Aux}_{\mathsf v}$'s are i.i.d.\ Bernoulli variables so that they take value 1 with probability $\mathsf p_{\mathrm{aux}} = 1- e^{-\mathsf c_{\mathrm g} \mathsf q/{250} }$ (note that the same trick was also used in \cite{Pis96}). We denote by $\mathsf P_{\mathrm{aux}}$ the law for $\mathrm {Aux}$, and we let $\vfk+ = \fk+ \times\mathsf P_{\mathrm{aux}}$ and $\vfkh+ = \fkh+ \times \mathsf P_{\mathrm{aux}}$ be the product measures on configurations for $\omega \times \mathrm{Aux} \in \Pi = \Omega \times \{0, 1\}^{\Lambda_{\mathsf N}}$.

 We say a vertex $\mathsf v$ is \emph{blue} if either $\mathsf v$ is bad or $\mathrm {Aux}_{\mathsf v} = 0$; otherwise, we say $\mathsf v$ is \emph{red}. We say $\mathsf B$ is a blue boundary if
$\mathsf B$ is blue (i.e., every vertex in $\mathsf B$ is blue), $\psi(\mathsf B)$ is a connected set containing the origin $\mathsf o$, and $\pari \psi(\mathsf B) = \mathsf B$. So in particular, $\{\mathsf o\}$, the set of origin itself, is a blue boundary as long as $\mathsf o$ is blue. We say $\mathsf B$ is the outmost blue boundary if $\mathsf B$ is a blue boundary and $\psi(\mathsf B') \subset \psi(\mathsf B)$ for any other blue boundary $\mathsf B'$.  We use the convention that the outmost blue boundary is $\emptyset$ if there exists no blue boundary. For convenience, we let $\mathtt{Blue} \subset \Lambda_{\mathsf N}$ be the collection of blue vertices and define $\mathtt{Red} = \Lambda_{\mathsf N} \setminus \mathtt{Blue}$. For $\mathsf B, \mathsf B'\subset \Lambda_{\mathsf N}$, let $\Pi_{\mathsf B, \mathsf B'}$ be the collection of $\omega \times \mathrm{Aux}\in \Pi$ such that the following hold (see Figure~\ref{fig:my_label} for an illustration):
\begin{itemize}
\item $\mathsf B$ is the outmost blue boundary.
\item If $\mathsf B \neq \emptyset$, then $\mathsf B'$ is the collection of all blue vertices which can be $2\mathsf k$-connected to $\mathsf B$ via blue vertices; if $\mathsf B = \emptyset$, then $\mathsf B'$ is the collection of blue vertices which are $2\mathsf k$-connected to $\mathsf o$ in blue vertices. (Actually we just fix $\mathsf k=4$ but keep the notation only for conceptual clearness.)
\end{itemize}
Let $\mathfrak B$ be the collection of pairs $(\mathsf B, \mathsf B')$ such that $\Pi_{\mathsf B, \mathsf B'} \neq \emptyset$. We observe that for $(\mathsf B, \mathsf B')\in \mathfrak B$ and for $\omega \times \mathrm{Aux} \in \Pi_{\mathsf B, \mathsf B'}$, we have that 
\begin{equation}\label{eq-B-prime-blue}
\mathsf B \subset \mathtt{Blue} \mbox{ and } \mathrm{Ball}(\mathsf B \cup \mathsf B', 2\mathsf k) \cap \mathtt{Blue} = \mathsf B \cup\mathsf B'\,,
\end{equation}
and that 
\begin{equation}\label{eq-red-path}
    \begin{aligned}
    &\text{for each point } \mathsf v\in\pare\psi(\mathsf B)\text{ there exists a red path in }\\
    &\mathrm{Ball}(\mathsf B\cup \mathsf B',2\mathsf k)\setminus\psi(\mathsf B) \text{ connecting }\mathsf v \text{ and } \pari(\mathrm{Ball}(\psi(\mathsf B)\cup\mathsf B',2\mathsf k)).
    \end{aligned}
\end{equation}
We may verify \eqref{eq-red-path} by contradiction. Since all vertices in $\pare \psi (\mathsf B)$ are red due to the fact that $\psi(\mathsf B)$ is the outmost blue boundary, if \eqref{eq-red-path} fails then there can be no red path joining some $\mathsf v\in\pare \psi(\mathsf B)$ and $\pare \Lambda_{\mathsf N}$, contradicting with the maximality of $\psi(\mathsf B)$. In addition, note that for $(\mathsf B, \mathsf B')\in \mathfrak B$, \eqref{eq-red-path} is a consequence of \eqref{eq-B-prime-blue}. 
By \cite[Lemma 2.1]{DP96} (see also \cite[Theorem 4]{Timar13} for a more general result, and see \cite{Kesten84, Hammond06} for related results) on boundary connectivity, we see that for $\mathsf k \geq 2$ there exists a connected subset $\mathtt R \subset  (\mathrm{Ball}(\mathsf B \cup \mathsf B', 2\mathsf k) \setminus \mathrm{Ball}(\mathsf B \cup \mathsf B', \mathsf k))$ such that 
\begin{equation}\label{eq-red-surounding}
\psi(\mathrm{Ball}(\mathsf B \cup \mathsf B', \mathsf k)) \subset \psi(\mathtt R) \mbox{ and } \mathtt R \subset \mathtt{Red}
\end{equation}
where the second inclusion follows from \eqref{eq-B-prime-blue}. Finally, we have that 
\begin{equation}\label{eq-R-to-Lambda-N}
\mbox{there exists a red path  connecting $\mathtt R$ and $\pari \Lambda_{\mathsf N}$}\,.
\end{equation}
This is true since otherwise by $\psi(\mathrm{Ball}(\mathsf B \cup \mathsf B', \mathsf k)) \subset \psi(\mathtt R)$ we cannot have $\pare\psi(\mathsf B) $ connected to $\pari \Lambda_{\mathsf N}$ via a red path.

\begin{figure}
    \centering
    \vspace{-5em}
    \includegraphics[width=35em]{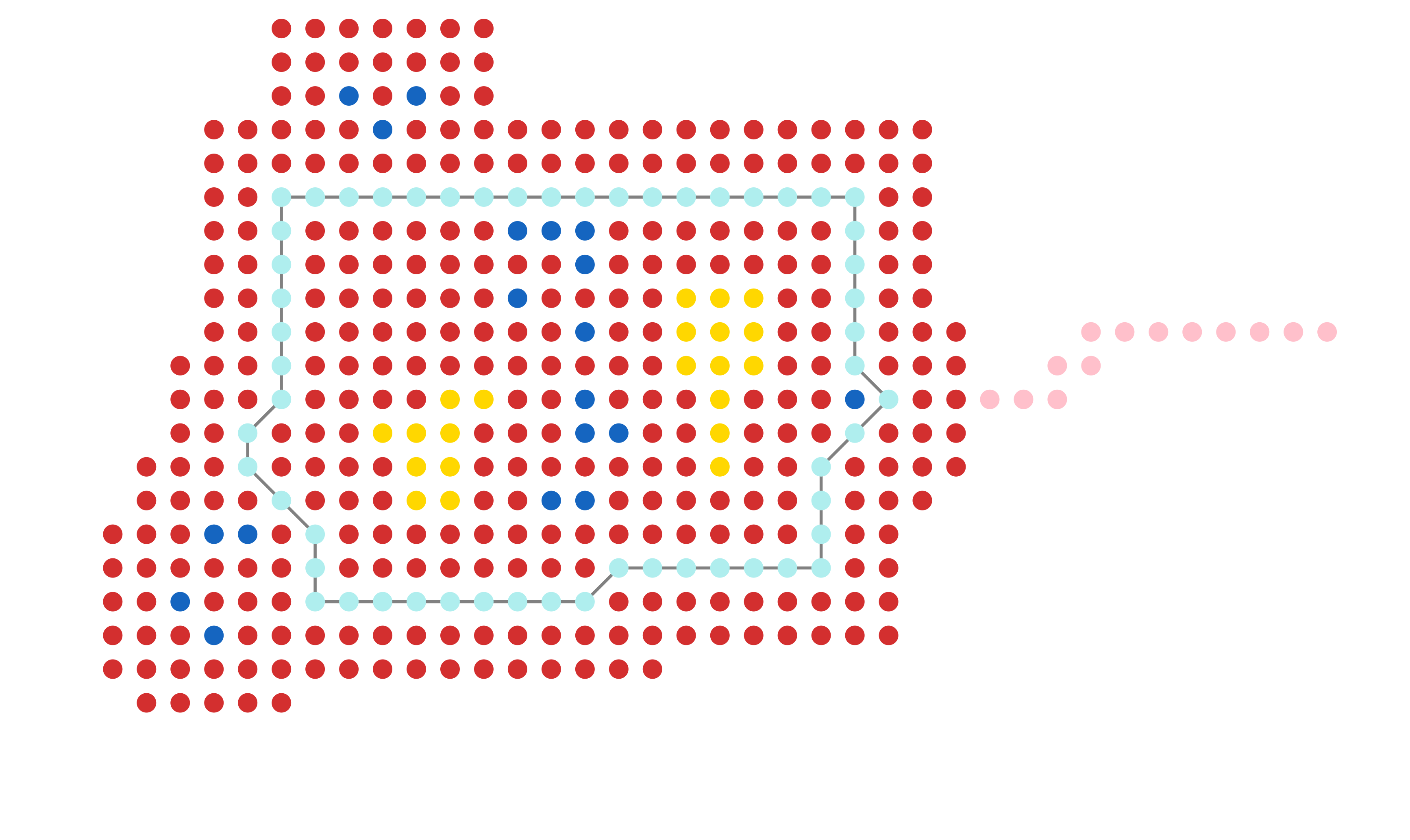}
    \vspace{-4em}
    \caption{An illustration of $\mathrm{Ball}(\mathsf{B\cup \mathsf B'},2)$. The picture depicts a possible pair $(\mathsf B, \mathsf B')\in \mathfrak B$ with $\mathsf k=1$ as an example. Precisely, $\mathsf{B}$ is the collection of light blue points (the outmost blue boundary for every $\omega \times \mathrm{Aux}\in \Pi_{\mathsf B, \mathsf B'}$) and $\mathsf{B'}$ is the collection of the dark blue points. In addition, the remaining points in $\mathrm{Ball}(\mathrm{B\cup B'},2)$ are red, the pink points denote a red path connecting $\pare\psi(\mathrm{Ball}(\mathsf{B\cup B'},2))$ and the boundary $\pari\Lambda_{\mathsf N}$, and the yellow points illustrate the `holes' defined in Definition~\ref{def-B-B-prime-U} below.}
    \label{fig:my_label}
\end{figure}

 Let $\mathcal C_o$ be the FK-cluster of the origin. By \eqref{eq-red-surounding} and by the definition of red (in fact, by the definition of good boxes), we see that for $\omega \times \mathrm{Aux} \in \Pi_{\mathsf B, \mathsf B'}$, 
\begin{equation}
\mbox{if } o \mbox{ is not $\omega$-connected to } \pare \Lambda_N \mbox{ and } \mathsf B \neq \emptyset, \mbox{ then } \mathcal C_o \subset \mathsf Q_{\psi(\mathrm{Ball}(\mathsf B \cup \mathsf B', \mathsf k)) }\,.
\end{equation}
In addition, we have that
\begin{equation}\label{eq-decomposition-disconnected-event-big}
\mbox{if } o \mbox{ is not $\omega$-connected to } \pare \Lambda_N \mbox{ and } \mathsf B  =\emptyset, \mbox{ then } \mathcal C_o \subset \mathsf {Q_o} \,,
\end{equation}
where both $o$ and  $\mathsf o$ are the vertex $(0, 0, 0)$ but conceptually $o$ is regarded as the origin in $\Lambda_N$ and $\mathsf o$ is regarded as the origin in the coarse grained box $\Lambda_{\mathsf N}$.

\subsection{Proof of Proposition~\ref{prop:main-thm-restatement}}\label{sec:sketch-of-proof}

The following is simply a rewrite of \eqref{eq:expressionstep1}:
    \begin{equation}\label{eq-decomposition-0-epsilon-field}
    \ism(\sigma_o = -1)=\sum_{(\mathsf B, \mathsf B') \in \mathfrak B}\sum_{o\in C_o, C_o \subset \Lambda_N}\vfkh +(\Pi_{\mathsf B, \mathsf B'}\cap \{\mcc C_o=C_o\})\times \frac{\exp(-\frac{\eps}{T}h_{C_o})}{2\cosh(\frac{\eps}{T}h_{C_o})}\,.
\end{equation} 
In light of \eqref{eq-decomposition-0-epsilon-field} (note that it obviously also holds for $\epsilon = 0$), the key is to compare the FK-measure without disorder to the FK-measure with disorder. As a warm up, we control the case of $\mathsf B = \mathsf B' = \emptyset$ which makes the major contribution to $\ism(\sigma_o = -1)$ but is technically simpler.
\begin{prop}\label{prop-compare-small-L}
For any $T<T_c$ and $\delta, \theta_1 > 0$, there exist $\mathsf q_0 = \mathsf q_0(T, \delta, \theta_1)$ such that for any $\mathsf q \geq \mathsf q_0$ the following holds. There exists $\eps_0=\eps_0(\mathsf q, \delta,\theta_1,T)$ such that for $\epsilon \leq \epsilon_0$, we have
 with $\mbb P$-probability at least $1-\delta$,
\begin{align*}
&\sum_{o\in C_o, C_o \subset \Lambda_N}\vfkh +(\Pi_{\emptyset, \emptyset}\cap \{\mcc C_o=C_o\})\frac{\exp(-\frac{\eps}{T}h_{C_o})}{2\cosh(\frac{\eps}{T}h_{C_o})} \\
&\leq e^{\theta_1}\sum_{o\in C_o, C_o \subset \Lambda_N}\vfk +(\Pi_{\emptyset, \emptyset}\cap \{\mcc C_o=C_o\})\times\frac{1}{2}\,.
\end{align*}
\end{prop}
Proposition~\ref{prop-compare-small-L} is plausible since one may expect that the influence from disorder only comes from disorder in $\mathsf Q_{\mathsf o}$ and thus it vanishes as $\epsilon \to 0$. However, this is not as obvious as one may think and in fact the proof uses the fact that the origin $\mathsf o \in \Lambda_{\mathsf N}$ is surrounded by good vertices. The much harder case is when $|\mathsf B|+ |\mathsf B'| \geq 1$ and especially when it is very large, although conceptually this should make very little contribution to $\ism(\sigma_o = -1)$. The following extension of results in \cite{Pis96} and \cite{Bodineau05} would serve as a useful ingredient for the analysis in this case as it gives us some room to spare for our estimates. 
\begin{prop}\label{prop:coarse-graining}
For any $T<T_c$, there exists $\mathsf q_0 = \mathsf q_0(T)$ such that for any $\mathsf q \geq \mathsf q_0$ the following holds for a constant $b=b(T, \mathsf q_0)>0$ (which does \emph{not} depend on $N$ or $\mathsf q$):
    \begin{equation}
       \vfk+(\bigcup_{(\mathsf B, \mathsf B')\in \mathfrak B: |\mathsf B \cup \mathsf B'| = L} \Pi_{\mathsf B, \mathsf B'})\leq \exp(-b \mathsf q L). \label{eq:coarse-graining}
    \end{equation}
\end{prop}
Thanks to Proposition~\ref{prop:coarse-graining}, we can afford an error that is exponentially large in the following proposition.
\begin{prop}\label{prop-large-surface-estimation}
For any $T<T_c$, $\delta > 0$, there exists $\mathsf q_0 = \mathsf q_0(T,\delta)$ such that for any $\mathsf q \geq \mathsf q_0$ the following holds. There exists $\eps_0=\eps_0(\mathsf q,\delta,T)$ such that for $\epsilon \leq \epsilon_0$, we have
 with $\mbb P$-probability at least $1- \delta$,
\begin{align*}
    \vfkh +(\Pi_{\mathsf B, \mathsf B'})\leq e^{200 \mathsf k^3 L}\times \vfk+(\Pi_{\mathsf B, \mathsf B'}) \mbox { for all } (\mathsf B, \mathsf B') \in \mathfrak B \mbox{ with } |\mathsf B| +|\mathsf B'| = L \geq 1\,.
\end{align*}
\end{prop}

\begin{proof}[Proof of Proposition~\ref{prop:main-thm-restatement}]
Write $\mathfrak B_L = \{(\mathsf B, \mathsf B') \in \mathfrak B: |\mathsf B| + |\mathsf B'| = L \}$. By  Proposition~\ref{prop-large-surface-estimation} we may choose $\mathsf q$ large and $\epsilon$ small enough such that $\mathbb P$-probability at least $1 - \delta$, 
\begin{align*}
&\sum_{L=1}^\infty \sum_{(\mathsf B, \mathsf B') \in \mathfrak B_L} \sum_{C_0}\vfkh +(\Pi_{\mathsf B, \mathsf B'}\cap \{\mcc C_o=C_o\})\times \frac{\exp(-\frac{\eps}{T}h_{C_o})}{2\cosh(\frac{\eps}{T}h_{C_o})} \leq \sum_{L = 1}^\infty \sum_{(\mathsf B, \mathsf B') \in \mathfrak B_L} \vfkh+(\Pi_{\mathsf B, \mathsf B'})\\
& \leq  \sum_{L = 1}^\infty \sum_{(\mathsf B, \mathsf B') \in \mathfrak B_L} e^{200 \mathsf k^3 L} \vfk+(\Pi_{\mathsf B, \mathsf B'}) \leq \sum_{L=1}^\infty e^{200\mathsf k^3L - b\mathsf q L}\,,
\end{align*}
where the second and third inequalities follow from Propositions~\ref{prop-large-surface-estimation} and \ref{prop:coarse-graining} respectively. Since the above bound tends to 0 as $\mathsf q \to \infty$, we complete the proof by combining Proposition~\ref{prop-compare-small-L} (which controls the case of $|\mathsf B| + |\mathsf B'| = 0$) and \eqref{eq-decomposition-0-epsilon-field}.
\end{proof}

\subsection{Influence from disorder}\label{subsec-influence-disorder}

In this subsection, we explain how to control the influence from disorder, that is, how to compare the measure with disorder to that without disorder. We first prove Proposition~\ref{prop-compare-small-L}, which is much simpler than  Proposition~\ref{prop-large-surface-estimation} but already encapsulates conceptual difficulties. 
\begin{proof}[Proof of Proposition~\ref{prop-compare-small-L}]
Recalling \eqref{eq-decomposition-disconnected-event-big}, on the event $\mathsf B = \emptyset$ and $\mathcal C_o \subset \Lambda_N$ (thus, $\mathcal C_o$ is not connected to $\pare \Lambda_N$) we have that the valid realization for $\mathcal C_o$ satisfies $C_o \subset \mathsf Q_{\mathsf o}$. Therefore, in this case we have that 
$|h_{C_o} |\leq \sum_{x\in \mathsf Q_o}|h_x|$. 
For any $\delta>0$, there exists a constant $H=H(\delta, \mathsf q)$
such that 
\begin{equation}
\mbb P(\sum_{x\in \mathsf Q_o}|h_x|<H)>1-\frac{\delta}{2}.
\end{equation}
On the event $\sum_{x\in \mathsf Q_o}|h_x|<H$, we have $|h_{C_o}| \leq  \sum_{x\in \mathsf Q_o}|h_x|< H$. Thus,
\begin{equation}
\frac{\exp(-\frac{\eps}{T}h_{C_o})}{2\cosh(\frac{\eps}{T}h_{C_o})}
=\frac{1}{1+\exp(\frac{2\eps h_{C_0}}{T})}
<\frac{1}{1+\exp(-\frac{2\eps H}{T})}<\frac{e^{\theta_1/2}}{2}\,,
\end{equation}
where the last inequality holds by choosing $\epsilon$ small enough depending on $(T, \theta_1, \mathsf q, H)$.
Therefore, it suffices to prove
\begin{equation}\label{eq-small-region-disorder-to-show}
\sum_{o\in C_o, C_o\subset \mathsf Q_{\mathsf o}}\vfkh +(\Pi_{\emptyset, \emptyset}\cap \{\mcc C_o=C_o\})\leq e^{\theta_1/2}\sum_{o\in C_o, C_o\subset \mathsf Q_{\mathsf o}}\vfk +(\Pi_{\emptyset, \emptyset}\cap \{\mcc C_o=C_o\}).
\end{equation}
To this end, let $\mathrm{Shell}(\mathsf o, (\mathsf k, 2\mathsf k]) = \mathrm{Ball}(\mathsf o, 2\mathsf k) \setminus \mathrm{Ball}(\mathsf o, \mathsf k)$.
Note that  $\{\mathsf B=\emptyset, \mathsf B'=\emptyset, \mcc C_o\subset \mathsf Q_{\mathsf o}\}$ is the intersection of the following three events:\\
(i)  $\mathrm{Ball}(\mathsf o, 2 \mathsf k) \subset \mathtt{Red}$;\\
(ii) $\pare \mathrm{Ball}(\mathsf o, 2 \mathsf k)  \xleftrightarrow{\mathtt{Red}} \pari \Lambda_{\mathsf N}$: there is a red path connecting $\pare \mathrm{Ball}(\mathsf o, 2 \mathsf k)$ and $\pari \Lambda_{\mathsf N}$;\\
(iii) $\mcc C_o\subset \mathsf Q_{\mathsf o}$.

Therefore, \eqref{eq-small-region-disorder-to-show} follows if one can show that for $\mathsf q$ sufficiently large and $\epsilon$ sufficiently small (which may depend on $\mathsf q$), we have that with $\mathbb P$-probability at least $1 - \frac{\delta}{2}$:
\begin{align}
&\vfk+(\mathrm{Ball}(\mathsf o, 2\mathsf k) \subset \mathtt{Red}, \pare \mathrm{Ball}(\mathsf o, 2 \mathsf k)  \xleftrightarrow{\text{red}} \pari \Lambda_{\mathsf N} \mid \mcc C_o\subset \mathsf Q_{\mathsf o}, \mathrm{Shell}(\mathsf o, (\mathsf k, 2\mathsf k]) \subset \mathtt{Red}) \geq e^{-\theta_1 /4} \,, \label{eq-small-to-show-1}\\
&\vfkh+( \mcc C_o\subset \mathsf Q_{\mathsf o}, \mathrm{Shell}(\mathsf o, (\mathsf k, 2\mathsf k]) \subset \mathtt{Red}) \leq e^{\theta_1/4} \vfk+( \mcc C_o\subset \mathsf Q_{\mathsf o}, \mathrm{Shell}(\mathsf o, (\mathsf k, 2\mathsf k]) \subset \mathtt{Red})\,. \label{eq-small-to-show-2}
\end{align}

\noindent\emph{Proof of \eqref{eq-small-to-show-1}.} By finite energy property, there exists a constant $c_1 = c_1(T) >0$ such that 
\begin{equation}\label{eq-C-0-lower-bound}
\vfk+(\mcc C_o\subset \mathsf Q_{\mathsf o}\mid \mathrm{Shell}(\mathsf o, (\mathsf k, 2\mathsf k]) \subset \mathtt{Red})) \geq c_1\,.
\end{equation} In addition, by \eqref{eq-domination} and $\mathsf p_{\mathrm{aux}}\to_{\mathsf q\to\infty}1$, we have 
$$\vfk+(\mathrm{Ball}(\mathsf o, 2\mathsf k) \subset \mathtt{Red}, \pare \mathrm{Ball}(\mathsf o, 2 \mathsf k)  \xleftrightarrow{\text{red}} \pari \Lambda_{\mathsf N} ) \to 1$$
as $\mathsf q \to \infty$. This completes the proof of \eqref{eq-small-to-show-1} by noting that  for $\mathsf q$ large enough, 
\begin{align*}
&\vfk+((\{\mathrm{Ball}(\mathsf o, 2\mathsf k) \subset \mathtt{Red}\} \cap \{\pare \mathrm{Ball}(\mathsf o, 2 \mathsf k)   \xleftrightarrow{\text{red}} \pari \Lambda_{\mathsf N}\})^c \mid \mcc C_o\subset \mathsf Q_{\mathsf o},
\mathrm{Shell}(\mathsf o, (\mathsf k, 2\mathsf k])\subset \mathtt{Red})) \\
& \leq 2 c_1^{-1} \vfk+((\{\mathrm{Ball}(\mathsf o, 2\mathsf k) \subset \mathtt{Red}\} \cap \{\pare \mathrm{Ball}(\mathsf o, 2 \mathsf k)   \xleftrightarrow{\text{red}} \pari \Lambda_{\mathsf N}\})^c)\,.
\end{align*}

\noindent \emph{Proof of \eqref{eq-small-to-show-2}.} 
Let $\mathsf V_{\mathrm g}$ be the set of good vertices in $\Lambda_{\mathsf N}$. Let $\mathsf E$ be the collection of edges within $\mathsf Q_{\Lambda_{\mathsf N} \setminus  \mathrm{Ball}(\mathsf o, \mathsf k)}$ and let $\Omega^{\mathsf E} = \{0, 1\}^{\mathsf E}$. Note that the event  $\{\mathrm{Shell}(\mathsf o, (\mathsf k, 2\mathsf k]) \subset \mathsf V_{\mathrm g}\}$ is measurable with respect to $\omega|_{\mathsf E}$. Let $\Omega^{\mathsf E}_{\mathrm g}$ be the collection of $\omega^{\mathsf E} \in \Omega^{\mathsf E}$ such that $ \mathrm{Shell}(\mathsf o, (\mathsf k, 2\mathsf k]) \subset \mathsf V_{\mathrm g}$ holds. Let $\Omega(\omega^{\mathsf E}) = \{\omega \in \Omega: \omega|_{\mathsf E} = \omega^{\mathsf E}\}$. Then
$$\fkh+(\mcc C_o\subset \mathsf Q_{\mathsf o}, \mathrm{Shell}(\mathsf o, (\mathsf k, 2\mathsf k]) \subset \mathsf V_{\mathrm g})=\sum_{\omega^{\mathsf E} \in \Omega^{\mathsf E}_{\mathrm g}}\fkh+(\Omega(\omega^{\mathsf E}))\times \fkh+(\mcc C_o\subset \mathsf Q_o \mid \omega|_{\mathsf E} = \omega^{\mathsf E})\,.$$
We claim that there exists  $\eps_0=\eps_0(\mathsf q, \delta/2,\theta_1,T)$ such that for any $\eps<\eps_0$ the following holds with $\mathbb P$-probability at least $1-\delta/2$: for any  $\omega^{\mathsf E}\in \Omega^{\mathsf E}_{\mathrm g}$,
    \begin{equation}\label{eq:small-comparison-2}
        \fkh+(\mcc C_o\subset \mathsf Q_{\mathsf o}\mid \omega|_{\mathsf E} = \omega^{\mathsf E})\leq e^{\theta_1/12}\fk+(\mcc C_o\subset \mathsf Q_{\mathsf o} \mid \omega|_{\mathsf E} = \omega^{\mathsf E}).
    \end{equation}
We will give a detailed proof for a substantially harder version of \eqref{eq:small-comparison-2} in the proof of Proposition~\ref{prop-large-surface-estimation}. Thus, here we only provide an explanation on why \eqref{eq:small-comparison-2} holds without a formal proof.  Note that on $\omega|_{\mathsf E} = \omega^{\mathsf E}$, there is a unique mesh cluster in $\mathsf Q_{\mathrm{Shell}(\mathsf o, (\mathsf k, 2\mathsf k])}$ and as a result edges outside $\mathsf E$ cannot join two clusters in $\mathsf Q_{\Lambda_{\mathsf N} \setminus \mathrm{Ball}(\mathsf o, \mathsf k)}$ (with respect to $\omega^{\mathsf E}$) which both have diameters larger than $\mathsf q$. Thus, the influence from disorder only comes from the disorder in $\mathsf Q_{\mathrm{Ball}(\mathsf o, 2\mathsf k)}$ which vanishes as $\epsilon \to 0$. By \eqref{eq-domination}, we have 
\begin{equation}\label{eq-small-comparison-good-prob-close-to-1}
\fk+( \mathrm{Shell}(\mathsf o, (\mathsf k, 2\mathsf k]) \subset \mathsf V_{\mathrm g}) \to 1
\end{equation} as $\mathsf q\to \infty$. Therefore, with $\mathbb P$-probability at least $1-\delta/2$
 \begin{align*}
 &\fkh+( \mcc C_o\subset \mathsf Q_{\mathsf o}, \mathrm{Shell}(\mathsf o, (\mathsf k, 2\mathsf k]) \subset \mathsf V_{\mathrm g})  \leq \max_{\omega^{\mathsf E} \in \Omega^{\mathsf E}_{\mathrm g}}  \fkh+( \mcc C_o\subset \mathsf Q_{\mathsf o} \mid \omega|_{\mathsf E} = \omega^{\mathsf E})  \\
 &\stackrel{\eqref{eq:small-comparison-2}}{\leq}  e^{\theta_1/12}\max_{\omega^{\mathsf E} \in \Omega^{\mathsf E}_{\mathrm g}}  \fk+( \mcc C_o\subset \mathsf Q_{\mathsf o} \mid \omega|_{\mathsf E} = \omega^{\mathsf E}) \leq  e^{\theta_1/6} \fk+( \mcc C_o\subset \mathsf Q_{\mathsf o}) \\
 &\stackrel{\eqref{eq-C-0-lower-bound},\eqref{eq-small-comparison-good-prob-close-to-1}}{\leq} e^{\theta_1/4}\fk+( \mcc C_o\subset \mathsf Q_{\mathsf o},  \mathrm{Shell}(\mathsf o, (\mathsf k, 2\mathsf k]) \subset \mathsf V_{\mathrm g})
 \end{align*}
 where we have applied the spatial mixing property (\cite[Corollary 1.4]{DCGR20}) in the third inequality and we have chosen $\mathsf q$ large in the last inequality. By the definition of red, this completes the proof of \eqref{eq-small-to-show-2}.
\end{proof}

Most of the essential difficulties in this paper lie in Proposition \ref{prop-large-surface-estimation}: while the very basic intuition is similar to that for Proposition~\ref{prop-compare-small-L} where we demonstrated that the influence from disorder comes from a small region, substantial additional challenges arise since now $\mathsf B \cup \mathsf B'$ can be very large and we have to show that influence from disorder is mainly from a region whose volume is proportional to $\mathsf B \cup \mathsf B'$. Actually, the somewhat delicate definition of $\mathsf {(B,B')}\in \mathfrak B$ was set exactly for this purpose. 
For notational convenience, we define
\begin{align*}
   \mathsf{S}_1=\mathrm{Ball}(\mathsf B\cup \mathsf B',\mathsf k)\setminus (\mathsf{B\cup B'}) \mbox{ and }\mathsf{S}_2=\mathrm{Ball}(\mathsf B\cup \mathsf B',2\mathsf k)\setminus\mathrm{Ball}(\mathsf B\cup \mathsf B',\mathsf k)\,.
\end{align*}
Then $\mathrm{Ball}(\mathsf B\cup \mathsf B',2\mathsf k)$ is the disjoint union of $\mathsf B, \mathsf B', \mathsf{S}_1, \mathsf{S}_2$, and \eqref{eq-B-prime-blue} is equivalent to 
 $\{\mathsf B\cup \mathsf B'\subset \mathtt{Blue}\} \cap \{\mathsf{S}_1\cup \mathsf{S}_2\subset \mathtt{Red}\}$. In addition, note that
 $\mathsf B\cup \mathsf B'\cup \mathsf{S}_1=\mathrm{Ball}(\mathsf B\cup \mathsf B',\mathsf k)$ is a connected set, since $\mathsf{B}\cup \mathsf B'$ is $2\mathsf k$-connected. Besides, the following (obvious) bounds hold for $\mathsf{S}_1,\mathsf{S}_2$:
\begin{equation}
    |\mathsf{S}_1|\leq [(2 \mathsf k+1)^3-1]|\mathsf{B\cup B'}| \mbox{ and }
    |\mathsf{S}_2|\leq [(4\mathsf k+1)^3-(2\mathsf k+1)^3]|\mathsf{B\cup B'}|\leq 80\mathsf k^3|\mathsf{B\cup B'}| .\label{eq:volume-of-S}
\end{equation}
Here recall that $\mathsf k=4$ thus $(4\mathsf k+1)^3=(17/4)^3\mathsf k^3<80\mathsf k^3.$

Since red is a typical event, one may be convinced that most of the probability cost is from the event that $\{\mathsf{B\cup B'}\subset\mathtt{Blue}\}$ and as a result we need to carefully compare its measure with disorder to that without disorder.  As hinted earlier, an important philosophy is to show that essentially the influence comes from disorder in a thin shell. To this end, let $\mathtt E$ be the collection of edges within
$\mathsf Q_{\Lambda_{\mathsf N}\setminus (\mathsf B \cup \mathsf B'  \cup \mathsf{S}_1)}$. Write $\Omega^{\mathtt E} = \{0, 1\}^{\mathtt E}$ and define 
\begin{equation}\label{eq-def-Omega-B-B-prime-E}
 \Omega_{\mathsf B, \mathsf B'}^{ \mathtt E} = \{\omega^{\mathtt E} \in \Omega^{\mathtt E}:  \mathsf{Q_v} \mbox{ is good in }\omega \mbox{ for all }\mathsf v\in \mathsf{S_2} \mbox{ if }\omega|_\mathtt E = \omega^{\mathtt E}\}\,.
 \end{equation}
 
  Note that
 $\{\mathsf{S}_2\subset \mathtt{Red}\} = \{\mathsf v \mbox{ is good  and } \mathrm{Aux}_{\mathsf v} = 1\mbox{ for all } \mathsf v\in \mathsf{S}_2\}$.
The  event that each vertex in $\mathsf{S}_2$ is good is measurable with respect to the  $\sigma$-algebra of $\Omega^{\mathtt E}$, therefore it holds for every $\omega \in \Omega(\omega^{\mathtt E})$ with  $\omega^\mathtt E\in \Omega_{\mathsf B, \mathsf B'}^{\mathtt E}$ where $\Omega(\omega^{\mathtt E})$ is defined by
\begin{equation}\label{eq-def-Omega-omega-E}
\Omega(\omega^{\mathtt E})= \{\omega\in \Omega: \omega|_{\mathtt E}=\omega^{\mathtt E}\}\,.
\end{equation}  
For the event on the auxiliary variables, we recall the connections between red/blue, good/bad and the auxiliary variables and we define
 $$\mathcal A_{\mathsf B, \mathsf B'}=\{\mathrm{Aux}\in \{0,1\}^{\Lambda_{\mathsf N}}:\mathrm{Aux}_{\mathsf v}= 1\, \mbox{ for all } \mathsf v\in \mathsf{S}_2\}\,.$$
Then the event $\{\mathsf{S}_2\subset \mathtt{Red}\}$ happens for all $\omega \times \mathrm{Aux}$ with $\omega \in \Omega(\omega^\mathtt E)$,  $\omega^\mathtt E\in \Omega_{\mathsf B, \mathsf B'}^{\mathtt E}$ and $\mathrm{Aux}\in\mathcal A_{\mathsf{B,B'}}$. 
 
Next we would like to decompose the event $\{\mathsf B \cup \mathsf B' \subset\mathtt{Blue}\}$ according to different $\omega^{\mathtt E}$. 
For each $\mathrm{Aux}\in \mathcal A_{\mathsf B, \mathsf B'}$ and $\omega^{\mathtt E} \in \Omega^{\mathtt E}_{\mathsf B, \mathsf B'}$ we define $\Omega_{\mathrm {Aux}}(\omega^\mathtt E) $ to be the collection of $\omega\in \Omega(\omega^\mathtt E)$ such that (note that $\mathsf{Q_v}$ is good for $\mathsf v\in \mathsf{S}_2$ on all $\omega\in \Omega(\omega^\mathtt E)$)
\begin{align}\label{eq-def-Omega-Aux}
 \mathsf {Q_v}\text{ is bad if }\mathsf{v}\in \mathsf B\cup \mathsf B'\text{ and }\mathrm{Aux}_{\mathsf v}= 1\,.
\end{align}
Therefore,
$\{\mathsf{S}_2\subset \mathtt{Red}\}\cap\{\mathsf B\cup \mathsf B'\subset\mathtt{Blue}\}$ occurs on $\omega \times \mathrm{Aux}$ for all $\omega \in \Omega_{\mathrm {Aux}}(\omega^\mathtt E)$ and $\mathrm{Aux}\in\mathcal A_{\mathsf{B,B'}}$.

Now, we have for all $\epsilon$ (which in particular includes the case for $\epsilon = 0$)
\begin{align}
        &\vfkh+(\{\mathsf{S}_2\subset \mathtt{Red}\}\cap\{\mathsf B\cup \mathsf B'\subset\mathtt{Blue}\})\nonumber\\
    =&\sum_{\omega^\mathtt E\in\Omega^\mathtt E_{\mathsf B, \mathsf B'}}\vfkh+(\Omega(\omega^\mathtt E)\cap \{\mathsf{S}_2\subset \mathtt{Red}\}\cap \{\mathsf B\cup \mathsf B'\subset\mathtt{Blue}\})\nonumber\\
    =&\sum_{\omega^\mathtt E\in\Omega^\mathtt E_{\mathsf B, \mathsf B'}}\sum_{\mathrm{Aux}\in \mathcal{A}_{\mathsf B, \mathsf B'}} \mathsf P_{\mathrm{aux}}(\mathrm{Aux})\times\fkh+(\Omega_{\mathrm {Aux}}(\omega^\mathtt E)) \label{eq:localize-disorder-decomposition-configuration-version}
    \\
    =&\sum_{\omega^\mathtt E\in\Omega^\mathtt E_{\mathsf B, \mathsf B'}}\big(\fkh+(\Omega(\omega^\mathtt E))\times\sum_{\mathrm{Aux}\in \mathcal{A}_{\mathsf B, \mathsf B'} }\fkh+(\Omega_{\mathrm{Aux}}(\omega^\mathtt E) \mid \omega|_{\mathtt E} = \omega^\mathtt E)\times  \mathsf P_{\mathrm{aux}}(\mathrm{Aux})\big)\,,\label{eq:localize-disorder-decomposition-conditional-probability-version}
\end{align}
where $
\fkh+(\Omega_{\mathrm{Aux}}(\omega^\mathtt E)\mid \omega|_\mathtt E = \omega^\mathtt E)=\frac{\fkh+(\Omega_{\mathrm{Aux}}(\omega^\mathtt E))}{\fkh+(\Omega(\omega^{\mathtt E}))}$ is the conditional probability given boundary condition of $\omega^{\mathtt E}$ on $\mathtt {E}$.

Roughly speaking, we would like to control the influence of disorder via comparing  terms in \eqref{eq:localize-disorder-decomposition-conditional-probability-version} between cases with and without disorder (actually the disorder $\epsilon h$ will be  modified later according to certain rules into a new external field $\epsilon \tau_{\omega}(h)$ in order to facilitate this comparison). An important ingredient is to upper-bound  the ratio between
$\fkh+(\Omega_{\mathrm{Aux}}(\omega^\mathtt E)\mid \omega|_{\mathtt E} = \omega^\mathtt E)$ and
$\fk+(\Omega_{\mathrm{Aux}}(\omega^\mathtt E)\mid \omega|_{\mathtt E} = \omega^\mathtt E)$ for $\omega^\mathtt E \in \Omega^{\mathtt E}_{\mathsf B, \mathsf B'}$ and $\mathrm{Aux} \in \mathcal A_{\mathsf B, \mathsf B'}$. To this end, we next examine the boundary condition $\omega^\mathtt E$ more closely.  
\begin{definition}\label{def-B-B-prime-U}
For any $(\mathsf{B,B'})\in\mathfrak B$,
let $\mathfrak U=\{\mathsf U_1,\mathsf U_2,\cdots,\mathsf U_n\}$ be the collection of `holes' of $\mathsf B\cup \mathsf B'\cup \mathsf{S}_1$, that is, all the connected components of $\Lambda_{\mathsf N}\setminus(\mathsf B \cup \mathsf B'\cup \mathsf{S}_1)$ which are not connected to $\pari\Lambda_{\mathsf N}$. Let 
$U_j=\mathsf{Q}_{\mathsf U_j}$. In addition, let $\mathsf U_*$ be the connected component of $\Lambda_{\mathsf N}\setminus(\mathsf B \cup \mathsf B' \cup \mathsf{S}_1)$ which is connected to $\pari\Lambda_{\mathsf N}$ (where we identify $\pari\Lambda_{\mathsf N}$ as a single point), and we let $U_* = \mathsf Q_{\mathsf U_*}$.
\end{definition}
Since every $\mathsf U_j$ 
contains at least one point in $\mathsf{S_2}$, we get from 
 $\eqref{eq:volume-of-S}$ that
\begin{equation}
    \label{eq:number-of-U}
    n\leq |\mathsf{S}_2|\leq 80\mathsf k^3|\mathsf{B\cup B'}|.
\end{equation}
We now claim that each $\mathsf U_j$ is simply connected for $1\leq j\leq n$.
 Otherwise, assume  $\mathsf U_j$ is not simply connected for some $j$, that is, $\psi(\mathsf{U}_j)\setminus \mathsf{U}_j \neq \emptyset$. Since $\mathsf U_j$ and $\mathsf U_i$ are not connected for all $i \neq j$ and since $\mathsf B\cup \mathsf B'\cup \mathsf{S}_1$ is connected, there must exist  $\mathsf v\in\psi(\mathsf{U}_j)\setminus \mathsf{U}_j$ such that $\mathsf v\in \mathsf{B\cup B'\cup S_1}$. Thus, $\mathsf v\in\psi(\mathsf U_j)$, implying that $\mathsf{B\cup B'\cup S_1}\subset \psi(\mathsf U_j)$.  Therefore $\mathsf U_j$ is the connected component  which is connected to $\pari\Lambda_{\mathsf N}$, arriving at a contradiction.

By the preceding claim,
 each $U_j$ is simply connected. 
    In addition, for each $\mathsf U_j$, by by \cite[Lemma 2.1]{DP96} (see also \cite[Theorem 4]{Timar13}) we can take a connected set 
    $\partial_{\mathsf k}\mathsf U_j \subset \{\mathsf u\in\mathsf U_j:d_{\infty}(\mathsf u,\pare \mathsf U_j)\leq \mathsf k \}$ such that $ \{\mathsf u\in\mathsf U_j:d_{\infty}(\mathsf u,\pare \mathsf U_j) > \mathsf k \} \subset \psi(\partial_{\mathsf k}\mathsf U_j)$. Further, let $\partial_{\mathsf k}U_j =\mathsf Q_{\partial_{\mathsf k}\mathsf U_j}$.
    By definition, we know $\partial_{\mathsf k}\mathsf U_j\subset \mathsf{S}_2$. So, for $\omega\in\Omega(\omega^\mathtt E)$ with $\omega^{\mathtt E}\in\Omega^{\mathtt E}_{\mathsf{B,B'}}$,  each box $\mathsf{Q_v}$ for $\mathsf v\in\partial_{\mathsf k} \mathsf U_j$ is good. Thus, for each $j=1, \ldots, n$, on $\omega|_{U_j}$, 
    \begin{equation}\label{eq-def-C-from-U}
    \begin{split}
    \mbox{there is a unique cluster $C_j=C_j(\omega)\subset U_j$ so that $C_j|_{\mathsf Q_{\partial_{\mathsf k} \mathsf U_j}}$
     has a component with}\\
     \mbox{ diameter $\geq \mathsf q$
and also $C_j|_{\mathsf Q_{\partial_{\mathsf k} \mathsf U_j}}$ has a unique component with diameter $\geq \mathsf q$.}
    \end{split}
    \end{equation}
     (Note that different $C_j$'s might be connected via open edges in $\omega$ in the whole $\Lambda_N$,
      and we will treat such complication in Section~\ref{sec:localization-disorder} via the mapping  $\Psi: \mathfrak C^{\mathtt E} \mapsto \mathfrak C$ to be introduced before \eqref{eq-structure-C-*}). For $\mathsf U_*$, we take $\partial_{\mathsf k}\mathsf U_*$ similarly but in the outer layer of $\mathsf U_*$. To be precise, we take  a connected set $ \partial_{\mathsf k}\mathsf U_* \subset \mathrm{Ball}(\mathsf{B\cup B'},2\mathsf k)\setminus\psi(\mathrm{Ball}(\mathsf{B\cup B'},\mathsf k))$ so that $\mathsf U_* \subset \psi(\partial_{\mathsf k}\mathsf U_*)^c$. Note that  $\partial_{\mathsf k}\mathsf U_*\subset \mathsf {S_2}$. Clearly, on $\omega|_{\mathsf Q_{\mathsf U_*}}$, there also exists a unique cluster $C^{\mathtt E}_\diamond=C^{\mathtt E}_\diamond(\omega)\subset U_*$ (here we recall that $\pare\Lambda_N$ is all connected together via the wired boundary condition) whose restriction 
on $\mathsf Q_{\partial_{\mathsf k} \mathsf U_*}$
has a component with diameter  $\geq \mathsf q$ and in addition $C^{\mathtt E}_\diamond|_{\mathsf Q_{\partial_{\mathsf k} \mathsf U_*}}$ has a unique component with diameter  $\geq \mathsf q$. (Here the subscript $\diamond$ instead of $*$ indicates that $C^{\mathtt E}_\diamond$ may not be connected to $\pare \Lambda_N$; the superscript $\mathtt E$ is to emphasize that $C^{\mathtt E}_\diamond$ is within $\mathtt E$, which will be useful notation-wise in Section~\ref{sec:localization-disorder}.)
 In addition, while the realization for  $C^{\mathtt E}_\diamond$ and $\pare \lamn$  may depend on $\omega \in \Omega(\omega^{\mathtt E})$, the event $\{C^{\mathtt E}_\diamond\overset{}{\leftrightarrow}\partial_e \lamn\}$ only depends on $\omega^{\mathtt E}$ (since each $\mathsf v\in \mathsf S_2$ is good for all $\omega^{\mathtt E} \in \Omega_{\mathsf B, \mathsf B'}^{\mathtt E}$).

Recall that for $C\subset\lamn$ we denote by $h_C=\sum_{x\in C}h_x$. Since $|h_{C_j}|$'s may be substantially larger than $|\mathsf B| + |\mathsf B'|$, we see that when two of them have opposite signs the probability for them to be connected can be as small as $e^{-s}$ where $s \gg |\mathsf B| + |\mathsf B'|$. This calls for a Peierls mapping where we also flip the signs of disorder in certain regions; this is inspired by \cite{DZ21} but with additional complications in our context. To this end, let $\xi_j$ be the sign of $h_{C_{j}}$ for $j=1,\ldots,n$, let $\xi_\diamond$ be the sign of $h_{C^{\mathtt E}_\diamond}$, and define the mapping $\tau_{\omega^{\mathtt E}}: \mbb R^{\lamn}\rightarrow \mbb R^{\lamn}$ by
   \begin{equation}\label{eq-def-flipping}
    (\tau_{\omega^{\mathtt E}}(h))_x 
    =\left\{
    \begin{aligned}
    &\xi_j h_x
    \mbox{ for }x\in U_j, j=1\cdots, n, \mbox{ if }C^{\mathtt E}_\diamond\leftrightarrow\pare\lamn; \\ 
    &\xi_\diamond \xi_j h_x
    \mbox{ for }x\in U_j, j=1\cdots, n,\mbox{ if }C^{\mathtt E}_\diamond\nleftrightarrow \pare\lamn;\\
    & h_x 
    \mbox{ for } x\not\in \cup_{j=1}^n U_j.
    \end{aligned}\right.
    \end{equation}
    The preceding definition may be summarized as follows: if $C^{\mathtt E}_\diamond$ is connected to $\pare \Lambda_N$ on $\omega^{\mathtt E}$, we flip the external field in each $U_j$ such that every $(\tau_{\omega^\mathtt E}(h))_{C_j}\geq 0$, where $(\tau_{\omega^\mathtt E}(h))_ C=\sum_{x\in  C}\tau_{\omega^\mathtt E}(h)_x$ (which is simlar to $h_C$); if $C^{\mathtt E}_\diamond$ is not  connected to $\pare \Lambda_N$, we flip the external field in each $U_j$  such that $(\tau_{\omega^{\mathtt E}}(h))_{C_j}$ has the same sign with $(\tau_{\omega^{\mathtt E}}(h))_{C^{\mathtt E}_\diamond}$. This corresponds to the definition in \eqref{eq-def-FK-external-field} and will be useful in Section~\ref{sec:localization-disorder}. In what follows, we will usually write $\tau_{\omega}$ instead of $\tau_{\omega^{\mathtt E}}$ for simplicity.

We denote the FK-measure with the external field $\tau_{\omega}(h)$ by $\fkf + {\tau_{\omega}(h)}$.
Similarly we have the notation  $\vfkf+{\tau_{\omega}(h)}$. We next control the influence of such flipping of disorder. 

\begin{lem}\label{lem:flip-external-field-estimation}
For all $\omega\in\Omega_{\mathrm{Aux}}(\omega^{\mathtt E})$ with $\omega^\mathtt E\in\Omega_{\mathsf{B,B'}}^\mathtt E$,
\begin{align}\label{eq:flip-external-field-estimation}
\fkf+{h}(\omega) \leq  \frac{\parf {\tau_{\omega}(h)}}{\parz} e^{c \epsilon H_{\mathsf B, \mathsf B'}} \fkf+{\tau_{\omega}(h)}(\omega) 
\end{align}
where we define $H_{\mathsf  B, \mathsf B'} = \sum_{v\in \mathsf Q_{\mathrm{Ball}(\mathsf B\cup \mathsf B', 4\mathsf k)}} |h_v|$ and $c=c(T)>0$ is a constant.
\end{lem}
The proof of Lemma~\ref{lem:flip-external-field-estimation} combines \eqref{eq-partition-function-equal}, \eqref{eq-def-FK-external-field} and a careful analysis of  clusters in $\omega$.
\begin{lem}\label{lem:comparez}
	There exists $c_{\mathrm z}>0$ such that with $\mbb P$-probability at least $1-\exp(-\frac{c_{\mathrm z}}{\eps})$ the following holds: for all $\mathsf {(B,B')}\in \mathfrak B$ and for all $\omega\in\Omega_{\mathrm{Aux}}(\omega^{\mathtt E})$ with $\omega^\mathtt E\in\Omega_{\mathsf{B,B'}}^\mathtt E$
	\begin{equation}
			\frac{\parf {\tau_{\omega(h)}}}{\parz} \leq e^{\frac{81 \mathsf k^3 \mathsf q^3\sqrt{\eps}}{T}\cdot |\mathsf B \cup \mathsf B'|} \mbox{ and } H_{\mathsf B, \mathsf B'} \leq \epsilon^{-1/2}\mathsf q^3 |\mathsf B \cup \mathsf B'|\,.\label{eq:compare-partition-z}
	\end{equation}
\end{lem}
Lemma~\ref{lem:comparez} is an extension of \cite{Chalker83, FFS84, DZ21} and its proof is included in Section~\ref{sec:coarse-graining-proof}.

In light of \eqref{eq:localize-disorder-decomposition-configuration-version} and \eqref{eq:flip-external-field-estimation}, a key ingredient is to compare $\fkf+{\tau_{\omega}(h)}(\Omega_{\mathrm{Aux}}(\omega^\mathtt E)\mid \omega|_{\mathtt E} = \omega^\mathtt E)$ with $\fk+(\Omega_{\mathrm{Aux}}(\omega^\mathtt E)\mid \omega|_{\mathtt E} = \omega^\mathtt E)$, as incorporated in the following two propositions. 

\begin{prop}\label{prop:inside-M-event-h-and-no-h}
    For any  $\delta>0$, there exists a constant $\mathsf q_0 = \mathsf q_0(\delta,T)$ large enough such that for any $\mathsf q>\mathsf q_0$, there exists an $\eps_1=\eps_1(\mathsf q,\delta,T)$ such that the following holds for any $\eps<\eps_1$. With $\mbb P$-probability at least $1-\delta$, for any $\mathsf{(B,B')}\in \mathfrak B$ and any $\omega^\mathtt E\in\Omega_{\mathsf{B,B'}}^\mathtt E $, we have
    \begin{align}
    \fkf+{\tau_\omega(h)}(\Omega_{\mathrm{Aux}}(\omega^{\mathtt E})\mid\omega|_{\mathtt E}= \omega^\mathtt E)\leq e^{100 \mathsf k^3|\mathsf B \cup \mathsf B'|}\fk+(\Omega_{\mathrm{Aux}}(\omega^{\mathtt E})\mid \omega|_{\mathtt E}=\omega^\mathtt E)\,.
    \label{eq:inside-M-event-h-and-no-h}
    \end{align}
\end{prop}

\begin{prop}\label{prop:inside-M-event-different-theta}
     For any  $\delta>0$, there exists a constant $\mathsf q_0 = \mathsf q_0(\delta,T)$ large enough  such that  for any $\mathsf q \geq \mathsf q_0$,  for every $\mathsf{(B,B')}\in \mathfrak B$ and every $\omega^{\mathtt E}, \tilde \omega^{\mathtt E} \in \Omega^{\mathtt E}_{\mathsf{B,B'}}$
    \begin{align}
    \fk+(\Omega_{\mathrm{Aux}}(\omega^{\mathtt E}) \mid \omega|_{\mathtt E} = \omega^\mathtt E)\leq e^{|\mathsf B \cup \mathsf B'|}\fk+(\Omega_{\mathrm{Aux}}(\tilde \omega^{\mathtt E})\mid \omega|_{\mathtt E} = \tilde\omega^\mathtt E )\,.\label{eq:inside-M-event-different-theta}
    \end{align}
\end{prop}

The underlying intuition of Proposition~\ref{prop:inside-M-event-h-and-no-h} is that the influence from disorder only comes from disorder in a thin shell around $\mathsf B \cup \mathsf B'$; the very construction of $(\mathsf B, \mathsf B')$ (so that they are surrounded by good vertices) and the choice of our Peierls mapping $\tau_{\omega}$ are both for the purpose of making this intuition into a rigorous proof.  Proposition~\ref{prop:inside-M-event-different-theta}  can be seen as a consequence of the spatial mixing property for the FK-Ising model. However, this is not obvious since this inequality is in the flavor of ratio spatial mixing where we are interested in configurations in a potentially very large space while boundary conditions are posed in places that are just of large constant distance away. Nevertheless, the statement is plausible since we are willing to lose an exponential factor. 

We need yet another estimate. To this end, let
 \begin{align*}
    \Pi_{\mathsf{B,B'}}^1&= \{\mathsf{S}_1\subset\mathtt{Red}\}\cap \{\pari\mathrm{Ball}(\mathsf{B\cup B'},2\mathsf k) \xleftrightarrow{\text{red}}\pari\Lambda_{\mathsf N}\}\,,\\
   \Pi_{\mathsf{B,B'}}^2    &=\{\mathsf B \cup \mathsf B'\subset\mathtt{Blue}\}\cap\{\mathsf{S}_2\subset\mathtt{Red}\}\,.
 \end{align*}    
\begin{prop}\label{prop:0-field-case-connecting-and-S2-red}
    There exists  $\mathsf q_0=\mathsf q_0(T,\mathsf k)$ such that for all $\mathsf q>\mathsf q_0$ and $\mathsf{(B,B')}\in\mathfrak B$
    \begin{equation}
        \vfk+(\Pi_{\mathsf{B,B'}}^1\mid \Pi_{\mathsf{B,B'}}^2)\geq e^{ - |\mathsf B \cup \mathsf B'|}.\label{eq:0-field-case-connecting-and-S2-red}
    \end{equation}
\end{prop}
Roughly speaking, the challenge in proving Proposition~\ref{prop:0-field-case-connecting-and-S2-red} is on the conditioning of $\Pi_{\mathsf{B,B'}}^2$ and especially on conditioning for appearance of blue vertices (since blue is rare). This challenge will be addressed by taking advantage of our auxiliary variables (see Section~\ref{sec:blue-red-percolation}). 

In addition, by \eqref{eq-domination} we have the following:
there exists  $\mathsf q_0=\mathsf q_0(T,\mathsf k)$ such that for any $\mathsf q>\mathsf q_0$ and $\mathsf{(B,B')}\in\mathfrak B$
\begin{equation}\label{eq:0-field-all-good-in-S1}
    \fk+(\mathsf{S}_2 \subset \mathsf V_{\mathrm g})\geq e^{-|\mathsf B \cup \mathsf B'|}\,,
\end{equation}
where we recall that $\mathsf V_{\mathrm g}$ is the collection of good vertices in $\Lambda_{\mathsf N}$.

\begin{proof}[Proof of Proposition~\ref{prop-large-surface-estimation}]
We fix an arbitrary $\omega_0^{\mathtt E} \in \Omega^{\mathtt E}_{\mathsf {B,B'}}$ in this proof. We first explain that $\Pi_{\mathsf{B,B'}}=\Pi_{\mathsf{B,B'}}^1\cap\Pi_{\mathsf{B,B'}}^2$. It suffices to check that $\mathsf B$ is the `outmost' blue boundary for  $\omega\in \Pi_{\mathsf{B,B'}}^1\cap\Pi_{\mathsf{B,B'}}^2$, i.e., every point in $\mathsf B$ is connected to $\pari\Lambda_{\mathsf N}$ by a red path. Recalling \eqref{eq-red-surounding}, we see that every path joining $\mathsf B$ and  $\pari \mathrm{Ball}(\mathsf{B\cup B'}, 2\mathsf k)$ intersects with $\mathtt R$. Combined with \eqref{eq-R-to-Lambda-N} and the fact that $\mathtt R$ is connected, this completes the verification of the claim. So we have
 $$
\vfkh+(\Pi_{\mathsf{B,B'}})\leq \vfkh+(\Pi_{\mathsf{B,B'}}^2) \mbox{ and }
\vfk+(\Pi_{\mathsf{B,B'}})=\vfk+(\Pi_{\mathsf{B,B'}}^2)\times \vfk+(\Pi_{\mathsf{B,B'}}^1\mid\Pi_{\mathsf{B,B'}}^2)\,.
$$
Let  $\mathsf q>\mathsf{q}_0$ be chosen large enough
 so that they satisfy requirements in Propositions~\ref{prop:inside-M-event-h-and-no-h}, \ref{prop:inside-M-event-different-theta}, \ref{prop:0-field-case-connecting-and-S2-red} and \eqref{eq:0-field-all-good-in-S1}. In addition, let $\eps_0$ be such that $81\mathsf k^3\mathsf q^3\sqrt{\eps_0}/T<1/2,c\sqrt{\eps_0}\mathsf q^3<1/2 ,e^{-c_{\mathrm z}/\eps_0} \leq \delta/2,$ (where $c_{\mathrm z}$ is as in Lemma~\ref{lem:comparez} and $c$ is as in Lemma \ref{lem:flip-external-field-estimation}) and that
 $\eps_0 \leq  \eps_1(\mathsf q,\delta/2,T)$
 in Proposition~\ref{prop:inside-M-event-h-and-no-h}.
Choosing $\eps\leq \eps_0$, by Lemmas~\ref{lem:comparez}, \ref{lem:flip-external-field-estimation} and Propositions~\ref{prop:inside-M-event-h-and-no-h}, \ref{prop:inside-M-event-different-theta}  we have that \eqref{eq:flip-external-field-estimation},  \eqref{eq:inside-M-event-h-and-no-h} simultaneously hold  with $\mbb P$-probability at least $1-\delta$. On this event, for all $(\mathsf B, \mathsf B')\in \mathfrak B$ with $|\mathsf B \cup \mathsf B'|  = L \geq 1$, we have:
 \begin{align}
   &\qquad \vfkh +(\Pi_{\mathsf{B,B'}})\leq \vfkh +(\Pi_{\mathsf{B,B'}}^2)\nonumber\\
    &\stackrel{\eqref{eq:localize-disorder-decomposition-configuration-version}}{=}\sum_{\omega^\mathtt E\in\Omega^\mathtt E_{\mathsf B, \mathsf B'}}\sum_{\mathrm{Aux}\in \mathcal{A}_{\mathsf B, \mathsf B'}}\mathsf P_{\mathrm{aux}}(\mathrm{Aux})\times\fkh+(\Omega_{\mathrm {Aux}}(\omega^\mathtt E)) \nonumber\\
    &\stackrel{\eqref{eq:flip-external-field-estimation}, \eqref{eq:compare-partition-z}}{\leq} e^{L}\sum_{\omega^\mathtt E\in\Omega^\mathtt E_{\mathsf B, \mathsf B'}}\sum_{\mathrm{Aux}\in \mathcal{A}_{\mathsf B, \mathsf B'}}\mathsf P_{\mathrm{aux}}(\mathrm{Aux})\times\fkf+{\tau_{\omega}(h)}(\Omega_{\mathrm {Aux}}(\omega^\mathtt E)) \nonumber\\
    &\stackrel{\eqref{eq:localize-disorder-decomposition-conditional-probability-version}}{=}e^{ L}\sum_{\omega^\mathtt E\in\Omega^\mathtt E_{\mathsf B, \mathsf B'}}\big(\fkf+{\tau_{\omega}(h)}(\Omega(\omega^\mathtt E))\times\hspace{-0.5em}\sum_{\mathrm{Aux}\in \mathcal{A}_{\mathsf B, \mathsf B'} }\fkf+{\tau_{\omega}(h)}(\Omega_{\mathrm{Aux}}(\omega^\mathtt E)\mid \omega|_{\mathtt E} = \omega^\mathtt E)\times  \mathsf P_{\mathrm{aux}}(\mathrm{Aux} )\big) \nonumber\\
    &\stackrel{\eqref{eq:inside-M-event-h-and-no-h}}{\leq} e^{ 101 \mathsf k^3 L}\sum_{\omega^\mathtt E\in\Omega^\mathtt E_{\mathsf B, \mathsf B'}}\big(\fkf+{\tau_{\omega}(h)}(\Omega(\omega^\mathtt E))\times\hspace{-0.5em}\sum_{\mathrm{Aux}\in \mathcal{A}_{\mathsf B, \mathsf B'} }\fk+(\Omega_{\mathrm{Aux}}(\omega^\mathtt E)\mid \omega|_{\mathtt E} = \omega^\mathtt E)\times  \mathsf P_{\mathrm{aux}}(\mathrm{Aux} )\big) \nonumber\\
    &\stackrel{\eqref{eq:inside-M-event-different-theta}}{\leq} e^{102 \mathsf k^3L}
    \sum_{\omega^\mathtt E\in\Omega^\mathtt E_{\mathsf B, \mathsf B'}}\big(\fkf+{\tau_{\omega}(h)}(\Omega(\omega^\mathtt E))\times\hspace{-0.5em}\sum_{\mathrm{Aux}\in \mathcal{A}_{\mathsf B, \mathsf B'} }\fk+(\Omega_{\mathrm{Aux}}(\omega_0^\mathtt E)\mid \omega|_{\mathtt E} = \omega_0^\mathtt E)\times  \mathsf P_{\mathrm{aux}}(\mathrm{Aux} )\big) \nonumber\\
    &= e^{102 \mathsf k^3L}\Big[\sum_{\mathrm{Aux}\in \mathcal{A}_{\mathsf B, \mathsf B'} }\fk+(\Omega_{\mathrm{Aux}}(\omega_0^\mathtt E)\mid \omega|_{\mathtt E} = \omega_0^\mathtt E)\times \mathsf P_{\mathrm{aux}}(\mathrm{Aux} )\Big]\times \hspace{-0.5em} \sum_{\omega^\mathtt E\in\Omega^\mathtt E_{\mathsf B, \mathsf B'}}\fkf+{\tau_{\omega}(h)}(\Omega(\omega^\mathtt E)). \label{eq-varphi-external-field-upper-bound}
    \end{align}
Recall that $\Omega_{\mathsf{B,B'}}^\mathtt E$ is the set of all the possible configurations on $\mathtt E$ such that $\mathsf{S}_2 \subset \mathsf V_{\mathrm g}$. Thus, 
$$\sum_{\omega^\mathtt E\in\Omega^\mathtt E_{\mathsf{B,B'}}}\fkh+(\Omega(\omega^\mathtt E))=\fkh+\big(
\bigcup_{\omega^\mathtt E\in\Omega^\mathtt E_{\mathsf{B,B'}}}\Omega(\omega^\mathtt E)\big)=\fkh+(\mathsf{S}_2 \subset \mathsf V_{\mathrm g}).$$
Note that the above does not quite hold for $\fkf+{\tau_{\omega^\mathtt E}(h)}$ since the Peierls mapping $\tau$ depends on $\omega^{\mathtt E}$, and thus it requires a slightly more careful analysis. Since for each $j = 1, \ldots, n$ the external field on (the whole) $U_j$ is flipped or not (and these are all the possible flippings involved), we see that the cardinality for $\mcc H = \{\tau_\omega: \omega|_{\mathtt E} \in \Omega^{\mathtt E}_{\mathsf B, \mathsf B'}\}$ is at most $2^n$. Recalling from \eqref{eq:number-of-U} $n\leq 80\mathsf k^3L$, we have
\begin{align}\label{eq-flipping-sum-upper-bound}
   \sum_{\omega^\mathtt E\in\Omega^\mathtt E_{\mathsf B, \mathsf B'}}\fkf+{\tau_{\omega^\mathtt E}(h)}(\Omega(\omega^\mathtt E))&\leq\sum_{\omega^\mathtt E\in\Omega^\mathtt E_{\mathsf B, \mathsf B'}}\sum_{\tau\in \mcc H}\fkf+{\tau(h)}(\Omega(\omega^\mathtt E))\nonumber\\
   &=\sum_{\tau \in\mcc H} \fkf+{\tau(h)} (\mathsf{S}_2 \subset \mathsf V_{\mathrm g})\leq 2^n\leq e^{80\mathsf k^3 L}.
\end{align}
We next lower-bound the measure without disorder:
\begin{align*}
    &\vfk+(\Pi_{\mathsf{B,B'}})=\vfk+(\Pi_{\mathsf{B,B'}}^2)\vfk+(\Pi_{\mathsf{B,B'}}^1\mid\Pi_{\mathsf{B,B'}}^2)\\
    \stackrel{\eqref{eq:0-field-case-connecting-and-S2-red}}{\geq}&e^{- L}\vfk+(\Pi_{\mathsf{B,B'}}^2)\\
    \stackrel{\eqref{eq:localize-disorder-decomposition-conditional-probability-version}}{=}&e^{- L}\sum_{\omega^\mathtt E\in\Omega^\mathtt E_{\mathsf B, \mathsf B'}}\big(\fk+(\Omega(\omega^\mathtt E))\times\sum_{\mathrm{Aux}\in \mathcal{A}_{\mathsf B, \mathsf B'} }\fk+(\Omega_{\mathrm{Aux}}(\omega^\mathtt E)\mid \omega|_{\mathtt E} = \omega^\mathtt E)\times  \mathsf P_{\mathrm{aux}}(\mathrm{Aux} )\\
    \stackrel{\eqref{eq:inside-M-event-different-theta}}{\geq}&e^{-2 L}\sum_{\omega^\mathtt E\in\Omega^\mathtt E_{\mathsf B, \mathsf B'}}\big(\fk+(\Omega(\omega^\mathtt E))\times\sum_{\mathrm{Aux}\in \mathcal{A}_{\mathsf B, \mathsf B'} }\fk+(\Omega_{\mathrm{Aux}}(\omega_0^\mathtt E)\mid \omega|_{\mathtt E} = \omega_0^\mathtt E)\times  \mathsf P_{\mathrm{aux}}(\mathrm{Aux} )\\
    =&e^{-2 L}\Big[\sum_{\mathrm{Aux}\in \mathcal{A}_{\mathsf B, \mathsf B'} }\fk+(\Omega_{\mathrm{Aux}}(\omega_0^\mathtt E)\mid \omega|_{\mathtt E} = \omega_0^\mathtt E) \times  \mathsf P_{\mathrm{aux}}(\mathrm{Aux} )\Big]\times \fk+(\mathsf{S}_2\subset \mathsf V_{\mathrm g})\\
    \stackrel{\eqref{eq:0-field-all-good-in-S1}}{\geq}&e^{-3 L}\sum_{\mathrm{Aux}\in \mathcal{A}_{\mathsf B, \mathsf B'} }\fk+(\Omega_{\mathrm{Aux}}(\omega_0^\mathtt E)\mid \omega|_{\mathtt E} = \omega_0^\mathtt E)\times  \mathsf P_{\mathrm{aux}}(\mathrm{Aux}).
\end{align*}
Combined with \eqref{eq-varphi-external-field-upper-bound} and \eqref{eq-flipping-sum-upper-bound}, this completes the proof of the proposition.
\end{proof}

\subsection{Method of coarse graining}\label{sec:coarse-graining-proof}

In this subsection we prove Proposition~\ref{prop:coarse-graining} and Lemma~\ref{lem:comparez}, both of which employ the method of coarse graining.

\begin{proof}[Proof of Proposition~\ref{prop:coarse-graining}]
Since $\mathsf B \cup \mathsf B'$ is $2\mathsf k$-connected and from $\mathsf B \cup \mathsf B'$ we can uniquely recover $(\mathsf B, \mathsf B')$ ($\mathsf B$ is the outmost contour that surrounds the origin in $\mathsf B \cup \mathsf B'$), we can control the enumeration of possible $(\mathsf B, \mathsf B')$ with $|\mathsf B\cup \mathsf B'| = L$ as follows: each such set is encoded by a contour (where the neighboring relation is $2\mathsf k$-neighboring) of length $2L$ (each contour corresponds to a depth-first-search process of a spanning tree of size $L$) where the starting point is in $\mathrm{Ball}(\mathsf o, 2\mathsf k L)$. It is straightforward that the number of ways for such encoding is at most $(2\mathsf k L+1)^3 (4\mathsf k+1)^{3L}$. Combined with \eqref{eq-domination}, this completes the proof. 
\end{proof}

\begin{proof}[Proof of Lemma~\ref{lem:comparez}]
In the case of $n = 1$, this was explicitly written in \cite{DZ21} with crucial input from \cite{Chalker83, FFS84}. Our current lemma has some minor technical complications and we now explain how to address them.

For each $U_1, \ldots, U_n$ and $I\subset \{1, \ldots, n\}$, let $h^I$ be obtained from $h$ by flipping the signs of disorder on $\cup_{i\in I} U_i$. The complication is that $\cup_{i=1}^n U_i$ is not connected. However,   since each $\mathsf U_i$ is connected to $\mathrm{Ball}(\mathsf{B\cup B'},2\mathsf k)$, we may add a path $P = P(U_1, \ldots, U_n)$ of length at most $80\mathsf k^3 L \mathsf q$ so that $\cup_{i\in I} U_I \cup P$ is a simply connected subset for all $I$. Let $h^{I, P}$ be obtained from $h$ by flipping the signs of disorder on $\cup_{i\in I} U_i \cup P$. Now applying \cite[Equation (18)]{DZ21} and \eqref{eq-partition-function-equal} (which relates the partition function for the FK-Ising model and to that for the Ising model), we get that 
\begin{equation}\label{eq-compare-partition-simply-connected}
\mbb P\big(\max_{(\mathsf B, \mathsf B') \in \mathfrak B_L}\frac{\mathcal Z^+_{\phi}(\epsilon h^{I, P})}{\mathcal Z^+_{\phi}(\epsilon h)} \geq  e^{\sqrt{\epsilon} \mathsf q^3 \mathsf k^3 L/T}\big) \leq e^{-2c_{\mathrm z} L/\epsilon}
\end{equation}
where $c_{\mathrm z} = c_{\mathrm z}(T) > 0$. Note that in \cite{DZ21} the bound on the ratio for partition functions is $e^{L/T}$, but it is straightforward to extend to $e^{L \sqrt{\epsilon}/T}$. Furthermore, by \cite[Lemma 3.1]{DZ21} (see also \cite[Equation (13)]{DZ21}), we see that for $\epsilon$ sufficiently small
\begin{align*}
\mbb P\big(\max_{(\mathsf B, \mathsf B') \in \mathfrak B_L}\frac{\mathcal Z^+_{\phi}(\epsilon h^{I, P})}{\mathcal Z^+_{\phi}(\epsilon h^I)} \geq e^{\sqrt{\epsilon} \mathsf k^3 \mathsf q L/T}\big) \leq \mbb P\big(\max_{P: |P| \leq  80\mathsf k^3 L \mathsf q} \sum_{v\in P} |h_v| \geq \epsilon^{-1/2} \mathsf k^3 \mathsf q L\big)  \leq e^{-2c_{\mathrm z} L/\epsilon}\,,
\end{align*}
where the second inequality follows from a union bound on $P$. Also, the similar probability estimates hold for $\{\min\limits_{(\mathsf B, \mathsf B') \in \mathfrak B_L}\frac{\mathcal Z^+_{\phi}(\epsilon h^{I, P})}{\mathcal Z^+_{\phi}(\epsilon h)} \leq e^{-\sqrt{\epsilon} \mathsf q^3 \mathsf k^3 L/T}\}$ and $\{\min\limits_{(\mathsf B, \mathsf B') \in \mathfrak B_L}\frac{\mathcal Z^+_{\phi}(\epsilon h^{I, P})}{\mathcal Z^+_{\phi}(\epsilon h^I)} \leq e^{-\sqrt{\epsilon} \mathsf k^3 \mathsf q L/T}\}$.
Combined with \eqref{eq-compare-partition-simply-connected} and the triangle inequality, this implies the inequality for the partition function via a simple union bound on $I$ (whose choice is at most $2^n$), on $n\leq L$ and on $L \geq 1$. 

For each $(\mathsf B, \mathsf B')$ with $|\mathsf B \cup \mathsf B'| = L$, we see from simple a Gaussian estimate that for sufficiently small $\epsilon$:
$$\mbb P(H_{\mathsf B, \mathsf B'} \geq \epsilon^{-1/2} \mathsf q^3 L) \leq e^{-\epsilon^{-1} L/100}\,.$$
Thus, the simultaneous bound for all $H_{\mathsf B, \mathsf B'}$ follows from a straightforward union bound once we have the upper bound on the enumeration of $(\mathsf B, \mathsf B')$ with $|\mathsf B\cup \mathsf B'| = L$ as in the proof of Proposition~\ref{prop:coarse-graining}.
\end{proof}

\subsection{Blue-red percolation}  \label{sec:blue-red-percolation}

In this subsection, we  prove Proposition~\ref{prop:0-field-case-connecting-and-S2-red}.
For convenience, in this subsection we will fix an arbitrary $(\mathsf{B,B'})\in\mathfrak B$ and denote by $B$ the event that $\mathsf B \cup \mathsf B'$ is blue, by $R_1, R_2$ the events that $\mathsf{S}_1, \mathsf{S}_2$ are red respectively, and by $R_{\mathrm c} = \{\pari \mathrm{Ball}(\mathsf{B\cup B'},2\mathsf k) \xleftrightarrow{\text{red}} \pari\Lambda_\mathsf N\}$.
Thus, we have 
$$\Pi_{\mathsf{B\cup B'}}^1= R_1 \cap R_{\mathrm c} \mbox{ and } \Pi_{\mathsf{B\cup B'}}^2=R_2\cap B\,.$$

As hinted earlier, the main challenge for the proof is to deal with the conditioning especially the conditioning of blue vertices. 
Recall that a vertex $\mathsf v$ is blue if either $\mathsf v$ is bad or $\mathrm {Aux}_{\mathsf v}= 0$,
and that the former occurs with probability at most $e^{-\mathsf c_{\mathrm g} \mathsf q }$ whereas the latter occurs with probability $e^{-\mathsf c_{\mathrm g} \mathsf q/{250} }$. Therefore, when $\mathsf q$ is large enough,  the conditioning of $\mathsf v$ being blue is roughly like the conditioning of $\mathrm{Aux}_{\mathsf v}= 0$; as a result, this has small influence on FK-percolation elsewhere.  To formalize this intuition, let 
$$\mathtt{Open}=\{\mathsf v\in\Lambda_{\mathsf N}:\mathrm{Aux}_{\mathsf v}= 1\} \mbox{ and } \mathtt{Closed}=\Lambda_{\mathsf N}\setminus\mathtt{Open}\,.$$
\begin{lem}\label{lem:blue-comes-from-closed}
We have that
    \begin{equation}\label{eq-blue-comes-from-bad}
        \vfk+(\mathsf{B\cup B'}\subset \mathtt{Closed}\mid B)\geq (1+e^{-\mathsf c_{\mathrm g}\mathsf q/250})^{-|\mathsf B\cup \mathsf B'|}.
    \end{equation}
\end{lem}
\begin{proof}
    On the one hand, by the definition of $\mathsf{Aux}$
    \begin{equation}\label{eq-closed-probability}
        \vfk+(\mathsf{B\cup B'}\subset \mathtt{Closed})=e^{-\mathsf c_{\mathrm g}\mathsf{q|B\cup B'|}/250}.
    \end{equation}       
    On the other hand, we could decompose the event $B=\{\mathsf B\cup \mathsf B' \mbox{ is blue}\}$ according to the realization of $\mathsf {Aux}$'s on $\mathsf B\cup \mathsf B'$ as follows:
    \begin{equation}
        B=\bigcup_{\mathsf M_1\cup \mathsf M_2 = \mathsf{B\cup B'}}\{\mathsf M_1\subset\mathtt{Closed}\}\cap\{\mathsf M_2\subset\mathtt{Open}\}\cap\{\mathsf v\text{ is bad for all } \mathsf{v\in M_2}\}\,,
        \label{eq:blue-decompose-into-bad-and-closed}
    \end{equation}
    where the disjoint union is over all the possible $\mathsf{M_1\cup M_2=B\cup B'}$ with $\mathsf{M_1\cap M_2}=\emptyset$.

For any $\mathsf{M_1\cup M_2=B\cup B'}$, we may choose $\mathsf M'_2 \subset \mathsf M_2$ such that (1) vertices in $\mathsf M'_2$ have pairwise $\ell_{\infty}$-distance at least $3$ and (2) $|\mathsf M'_2| \geq |\mathsf M_2|/125$ (note that any maximal $\mathsf M'_2$ satisfying (1) also satisfies (2)). Combined with \eqref{goodboxprobability}, it yields that
\begin{equation}
\fk+(\mathsf{v}\text{ is bad for all } \mathsf v\in \mathsf M_2)\leq \fk+(\mathsf{v}\text{ is bad for all } \mathsf v\in \mathsf M'_2)\leq \exp(-\mathsf c_{\mathrm g}\mathsf q |\mathsf M_2|/125).
\label{eq:many-bad-upper-bound}
\end{equation}
Writing $L = |\mathsf{B\cup B'}|$, we see that the number of choices for $\mathsf M_2$ with $|\mathsf M_2| = m$ is at most $\binom{L}{ m}$. Combining this with  \eqref{eq:blue-decompose-into-bad-and-closed} and \eqref{eq:many-bad-upper-bound}, we have
\begin{align}
\vfk+(B)\leq\sum_{\mathsf m=0}^L\binom{L}{m}e^{-(L-m)\mathsf c_{\mathrm g}\mathsf q/250}e^{- m \mathsf c_{\mathrm g}\mathsf q/125}= (e^{-\mathsf c_{\mathrm g}\mathsf q/250}+e^{-\mathsf c_{\mathrm g}\mathsf q/125})^L\,.
\end{align}
Combined with \eqref{eq-closed-probability}, this yields \eqref{eq-blue-comes-from-bad} by a simple Bayesian computation. 
    \end{proof}

We are now ready to prove Proposition~\ref{prop:0-field-case-connecting-and-S2-red}.

\begin{proof}[Proof of Proposition~\ref{prop:0-field-case-connecting-and-S2-red}]
Since $\{\mathsf{B\cup B'}\subset\mathtt {Closed}\} \subset B$, we have 
    \begin{align*}
\vfk+(R_1\cap R_2\cap R_{\mathrm c}\hspace{-0.1em}\mid\hspace{-0.1em} B)\geq \vfk+(\mathsf{B\cup B'}\subset\mathtt {Closed}\hspace{-0.1em}\mid\hspace{-0.1em} B) \vfk+(R_1\cap R_2\cap R_{\mathrm c}\hspace{-0.1em}\mid\hspace{-0.1em} \mathsf{B\cup B'}\subset\mathtt{Closed}).
    \end{align*}
Recalling that $\mathsf{B\cup B'}\subset\mathtt{Closed}$ is an event measurable with respect to $\mathrm{Aux}$,  we see that it is independent of the event $R_1\cap R_2\cap R_{\mathrm c}$. Thus,
    \begin{align*}
\vfk+(R_1\cap R_2\cap R_{\mathrm c}\mid \mathsf{B\cup B'}\subset\mathtt{Closed})=\vfk+(R_1\cap R_2\cap R_{\mathrm c}).
    \end{align*}
By \eqref{eq-domination}, we have that
$$\vfk+(R_1\cap R_2\cap R_{\mathrm c})\geq \mbb P_{\hat \rho}(R_1\cap R_2\cap R_{\mathrm c})\geq \mathsf p_{\mathsf q}^{\mathsf{|S_1\cup S_2|}}\mbb P_{\hat \rho}(R_{\mathrm c})$$
where $\hat \rho$ is a Bernoulli percolation with parameter $\mathsf p_{\mathsf q} \to 1$ as $\mathsf q \to \infty$.
Since $\mbb P_{\hat \rho}(R_{\mathrm c}) \to 1$ as $\mathsf p_{\mathsf q} \to 1$, we may choose $\mathsf q_0$ large enough so that for all $\mathsf q \geq \mathsf q_0$
\begin{align*}
    \vfk+(R_1\cap R_{\mathrm c} \cap R_2 \mid  B)\geq (1+e^{-\mathsf c_{\mathrm g}\mathsf q/250})^{-|\mathsf B\cup \mathsf B'|}\times \mathsf p_{\mathsf q}^{(4\mathsf k+1)^3\mathsf{|B\cup B'|}}\mbb P_{\hat \rho}(R_{\mathrm c})\geq e^{-|\mathsf B\cup \mathsf B'|}.
\end{align*}
where we used Lemma~\ref{lem:blue-comes-from-closed} and the fact that $|\mathsf{S}_1\cup \mathsf{S}_2|\leq (4\mathsf k+1)^3|\mathsf {B\cup B'}|$. Since $$ \vfk+(R_1\cap R_{\mathrm c} \mid  R_2\cap B)\geq \vfk+(R_1\cap R_{\mathrm c} \cap R_2 \mid  B),$$ this completes the proof of the proposition.
\end{proof}

\section{Spatial Mixing}\label{sec:spatial-mixing}

In this section, we prove Proposition~\ref{prop:inside-M-event-different-theta}.
As we mentioned earlier, the main obstacle comes from the fact that the distance between $\mathsf{Q_{B\cup B'}}$ and $\mathtt E$ is just a large constant, which can be substantially smaller than the diameter of $\mathsf{Q_{B\cup B'}}$. In order to address this, we review a random partitioning scheme in Section~\ref{sec:partition} from which we obtain a deterministic partition of $\mathsf B \cup \mathsf B'$ with some nice properties as in Section~\ref{sec:partition-B-B-prime}. In Section~\ref{sec:prop-spatial-mixing-proof} we apply a spatial mixing result in each set of the aforementioned partitioning and the proof of  Proposition~\ref{prop:inside-M-event-different-theta} follows by putting together estimates in all these small sets.

    \subsection{Random partition of finite metric space}\label{sec:partition}

Partition of metric space plays an important role in the study of metric geometry and theoretical computer science (e.g., it is an important pre-step in applying the computing scheme of ``divide and conquer''). Usually, it would be useful to have most of the points in the ``interior'' of the sets in the partitioning, since in many applications points that are near the boundaries of the sets in the partitioning cannot be ``controlled''. It turns out that introducing randomness in the partitioning scheme is a powerful method to generate such desired partitions, and in what follows we review a result from 
\cite{CKR05}. 
    
    Let $(X,d)$ be a finite metric space. If $P$ is a partition of $X$, for every $x\in X$ we denote by $P(x)$ the unique set in $P$ containing $x$. The partition $P$ is called $R$-bounded if the diameter of each set in $P$ is at most $R$. We say a random partition $\mcc P$ is $R$-bounded if $\mcc P$ is supported on $R$-bounded partitions of $X$. For $x\in X$ and $r>0$, we let $B(x, r) = \{y\in X: d(x, y)\leq r\}$ be the ball of radius $r$ centered at $x$. For positive numbers $s<t$, we define a function $H$ by
    $H(s,t)=\sum_{n=s+1}^t\frac{1}{n}$.
    Thus, for every $n\geq 2$, $H(1,n)=O(\log n)$.
    
    The result of \cite{CKR05} (see also \cite[Section 2]{KLM05}) can be rephrased in the following way which would be more suitable for our application. For every $R>0$, there exists an $R$-bounded random partition $\mcc P$ of set $X$ with law $\mathrm{P}$  such that  for every $x\in X$ and $0<r<\frac{R}{8}$, 
    \begin{equation}\label{eq:random-partition-boundary-estimation}
        \mathrm{P}(B(x,r) \not\subset \mcc P(x))\leq \frac{8r}{R}H(|B(x,\frac{R}{8})|,|B(x,R)|)\,.
    \end{equation}

   \subsection{Partition of explored clusters from outmost blue boundary}\label{sec:partition-B-B-prime}

    In this subsection, we apply \eqref{eq:random-partition-boundary-estimation} to the metric space $(X, d)$ with $X = \mathsf B \cup \mathsf B'$ and $d = d_{\infty}$. Let $R = \mathsf q^4$. By \eqref{eq:random-partition-boundary-estimation} (applied with $r = 16$), there exists an $R$-bounded random partition $\mcc P$ of $(X, d)$ with law $\mathrm{P}$ such that for every $\mathsf x\in X$,
	    \begin{equation}\label{eq:partition-of-edges-delta}
	        \mathrm{P}(B(\mathsf x, 16)\not\subset \mcc P(\mathsf x))\leq\frac{128}{R}H(|B(\mathsf x,\frac{R}{8})|,|B(\mathsf x,R)|) \leq
            c_1\times \frac{\ln R}{R},
	    \end{equation}
	   where $c_1>0$ is some constant and in the last inequality we used $|B(\mathsf x,R)|\leq (2R+1)^3$. Let $\mathsf P^{\mathrm b} = \{\mathsf x\in X: B(\mathsf x, 16) \not\subset \mcc P(\mathsf x)\}$. By \eqref{eq:partition-of-edges-delta} we see that $\mathrm {E} |\mathsf P^{\mathrm b}| \leq c_1\times \frac{\ln R}{R} |X|$ and thus there exists a deterministic partition $\mcc P_*=\{\mathsf P_1,\cdots, \mathsf P_m\}$ for some $m\geq 1$ such that
\begin{equation}\label{eq-partition-P-boundary-small}
|\mathsf P^{\mathrm b}_*| \leq c_1\times \frac{\ln R}{R} |X|\,.
\end{equation}
For each $j$, we divide $\mathsf P_j$ into the disjoint union of $\mathsf P_j^{\mathrm b}$ and $\mathsf P_j^{\mathrm o}$ where
    $$\mathsf P_j^{\mathrm o} =\{\mathsf x\in \mathsf P_j: B(\mathsf x,16)\subset \mcc P_*(\mathsf x)\} \mbox{ and }
    \mathsf P_j^{\mathrm b}=\{\mathsf x\in \mathsf P_j: B(\mathsf x,16)\not\subset \mcc P_*(\mathsf x)\}\,.$$
   Let $\mathsf P_*^{\mathrm o}=\bigcup_{j=1}^m \mathsf P_j^{\mathrm o} = X \setminus \mathsf P_*^{\mathrm b}$. 
 In Section~\ref{sec:prop-spatial-mixing-proof}, we wish to apply the weak mixing property for each $\mathsf P_j^{\mathrm o}$  and for configurations on edges within $\mathsf P_*^{\mathrm b}$ we will just use an obvious lower bound that comes from  the finite energy property (this will not cause too much error thanks to \eqref{eq-partition-P-boundary-small}). To this end, let $\mathsf W_j = \mathrm{Ball}(\mathsf P_j^{\mathrm o}, 4)$ where we recall that $\mathrm{Ball}$ is the ball in $\Lambda_{\mathsf N}$ with respect to $\ell_\infty$-distance. We claim that $d_{\infty}(\mathsf W_i, \mathsf W_j) > 4$ for $i\neq j$; this is true since otherwise by the triangle inequality we have $d_\infty(\mathsf P_j^{\mathrm o}, \mathsf P_j^{\mathrm o}) \leq 12$ which is a contradiction.
 Writing $W_j = \mathsf Q_{\mathsf W_j}$, we see that $W_j$'s are disjoint from each other (since $\mathsf Q_{\mathsf u} \cap \mathsf Q_{\mathsf v} = \emptyset$ if $d_{\infty}(\mathsf u , \mathsf v) > 1$).

	\subsection{Proof of Proposition~\ref{prop:inside-M-event-different-theta}}\label{sec:prop-spatial-mixing-proof}

Recall that $\Omega_{\mathrm{Aux}}(\omega^\mathtt E)$ is the set of configurations $\omega$ such that $\omega|_\mathtt E=\omega^\mathtt E$ and $\mathsf v$ is bad (with respect to $\omega$) for  $\mathsf v\in \mathsf B\cup \mathsf B'$ with $\mathrm{Aux}_{\mathsf v}=1$.
So there is a natural bijection $\chi=\chi_{_{\mathrm{Aux},\omega^{\mathtt E},\tilde\omega^\mathtt E}}:\Omega_{\mathrm{Aux}}(\omega^\mathtt E)\to\Omega_{\mathrm{Aux}}(\tilde\omega^\mathtt E)$ so that $\chi(\omega)|_{\mathtt E} = \tilde \omega^{\mathtt E}$ and $\chi(\omega)|_{\mathtt E^c} = \omega|_{\mathtt E^c}$. Let $L = |\mathsf B \cup \mathsf B'|>0$. 
It suffices to show that the following holds as long as $\mathsf q$ is sufficiently large (recall the definition of $\Omega_{\mathrm{Aux}}(\omega^\mathtt E)$):  for any $\omega^*\in\Omega_{\mathrm{Aux}}(\omega^\mathtt E)$
\begin{equation}\label{eq-spatial-mixing-to-show}
    \fk+(\omega|_{{\mathsf{B\cup B'}}}=\omega^*|_{{\mathsf{B\cup B'}}}\Big{|} \omega|_\mathtt E=\omega^\mathtt E)\leq e^{L}\fk+(\omega|_{{\mathsf{B\cup B'}}}=\chi(\omega^*)|_{{\mathsf{B\cup B'}}}\Big{|} \omega|_\mathtt E=\tilde\omega^\mathtt E)\,.
\end{equation}
(Note that by the definition of $\chi$, for any $\omega^*\in \Omega_{\mathrm{Aux}}(\omega^\mathtt E)$, we have $\omega^*|_{{\mathsf{B\cup B'}}}=\chi(\omega^*)|_{{\mathsf{B\cup B'}}}$.) Here and in what follows, for a set $\mathsf A\subset \Lambda_{\mathsf N}$, we use $\omega|_{\mathsf A}$ to denote $\omega|_{\mathsf{Q_A}}$ for notation convenience.

Taking advantage of our partition $\mathcal P_*$, we write
\begin{align*}
 \fk+\left(\omega|_{\mathsf{B\cup B'}}\hspace{-0.1em}=\hspace{-0.1em}\omega^*|_{\mathsf{B\cup B'}}\big{|}\ \omega_\mathtt E=\omega^\mathtt E\right)\hspace{-0.1em}=\hspace{-0.1em}& \prod_{j=1}^m\fk+\left(\omega_{\mathsf P_j^{\mathrm o}}=\omega^*|_{\mathsf P_j^{\mathrm o}}\Big{|}\ \omega|_\mathtt E\hspace{-0.1em}=\hspace{-0.1em}\omega^\mathtt E,\omega|_{\bigcup_{s=1}^{j-1}\mathsf P_s^{\mathrm o}}=\omega^*|_{\bigcup_{s=1}^{j-1}\mathsf P_s^{\mathrm o}}\right)\\ 
 &\times\fk+\left(\omega|_{\mathsf P_*^{\mathrm b}}=\omega^*|_{\mathsf P_*^{\mathrm b}}\Big{|} \omega|_\mathtt E=\omega^\mathtt E,\omega|_{\mathsf P_*^{\mathrm o}}=\omega^*|_{\mathsf P_*^{\mathrm o}}\right). 
\end{align*}
A similar expression holds for $\fk+\left(\omega|_{\mathsf{B\cup B'}}=\chi(\omega^*)|_{\mathsf{B\cup B'}}\Big{|}\   \omega|_\mathtt E=\tilde\omega^\mathtt E\right)$. By the definition of $\chi$, we have $\omega|_{\mathtt E^c} = \chi(\omega)|_{\mathtt E^c}$. Since $W_i$'s are disjoint,
the conditioning $\omega|_\mathtt E=\omega^\mathtt E, \omega|_{\bigcup_{s=1}^{j-1}\mathsf P_s^{\mathrm o}} = \omega^*|_{\bigcup_{s=1}^{j-1}\mathsf P_s^{\mathrm o}}$ can be seen as a boundary condition on $\pare W_j$, and thus we may write
\begin{equation}
\begin{split}
     &\fk+\left(\omega|_{\mathsf{B\cup B'}}=\omega^*|_{\mathsf{B\cup B'}}\Big{|}\ \omega|_\mathtt E=\omega^\mathtt E\right)=
     \prod_{j=1}^m\fk+\left(\omega|_{\mathsf P_j^{\mathrm o}}=\omega^*|_{\mathsf P_j^{\mathrm o}}\Big{|}\  \xi_j\right) \times\fk+\left(\omega|_{\mathsf P_*^{\mathrm b}}=\omega^*|_{\mathsf P_*^{\mathrm b}}\Big{|}\  \xi\right);\\
    &\fk+\left(\omega|_{\mathsf{B\cup B'}}= \omega^*|_{\mathsf{B\cup B'}}\Big{|}\  \omega|_\mathtt E=\tilde\omega^\mathtt E\right)=
    \prod_{j=1}^m\fk+\left(\omega|_{\mathsf P_j^{\mathrm o}}=\omega^*|_{\mathsf P_j^{\mathrm o}}\Big{|}\  \tilde \xi_j\right) \times\fk+\left(\omega|_{\mathsf P_*^{\mathrm b}}=\omega^*|_{\mathsf P_*^{\mathrm b}}\Big{|}\  \tilde \xi\right)\,.\label{eq:conditional-prob-two-parts}
\end{split}
\end{equation}
Here  $\xi_j,\xi, \tilde \xi_j, \tilde \xi$ are the boundary conditions induced by the events we condition on (as we pointed out above).
By the finite energy property, there exist constants $c_2 = c_2(T)>0$  such that
\begin{equation}\label{eq:boundary-edges-estimate}
\frac{\fk+\left(\omega|_{\mathsf P_*^{\mathrm b}}=\omega^*|_{\mathsf P_*^{\mathrm b}}\Big{|}\  \xi\right)}{ \fk+\left(\omega|_{\mathsf P_*^{\mathrm b}}=\omega^*|_{\mathsf P_*^{\mathrm b}}\Big{|}\  \tilde\xi \right)}\leq e^{c_2 \mathsf q^3|\mathsf P_*^{\mathrm b}|}\leq \exp(\frac{4 c_1c_2\ln \mathsf q}{\mathsf q}L)\,,
\end{equation}
where the last inequality follows from \eqref{eq-partition-P-boundary-small} (recall $R=\mathsf q^4$).

It remains to compare the measures of $\omega|_{\mathsf P_j^{\mathrm o}}$ with boundary condition $\xi_j$ to that with boundary condition $\tilde \xi_j$.
For convenience of exposition, we take $j=1$ in what follows and we also drop $j$ from the subscript for simplicity of notation (e.g., we write $W = W_1$).
 
The comparison method is a modification of \cite[Section 3.2]{DCGR20}, where the ratio spatial mixing was proved for events supported on $\Lambda_N$ with boundary conditions on $\Lambda_{2N}^c$ (see \cite[Corollary 1.4]{DCGR20}). For any boundary condition $\xi$ of $\pare W$, we wish to construct a coupling $\phi_{p,W}^{\xi, \mathbf 1}$ on pairs of configurations in $W$ which we denote by $(\omega^{\xi},\omega^{\mathbf 1})$ with $\phi_{p,W}^{\xi}$ being the law of the first marginal and $\phi_{p,W}^{\mathbf 1}$ being the law of the second marginal (here $\mathbf 1$ stands for the wired boundary condition). Fix 
$l$ to be a large constant as in \cite{DCGR20}  so that the block $\Lambda_\ell$ (a block is a box of side length $2l$) is good with probability sufficiently close to 1.
(After choosing $l$, we will then choose $\mathsf q$ as a sufficiently large constant so $\mathsf q \gg l$.) Following \cite{DCGR20}, we say a block $Q_v\in \mcc Q_l(W)$ is \emph{very good} in $(\omega^{\xi},\omega^{\mathbf 1})$ if it is a good block for both $\omega^{\xi}$ and $\omega^{\mathbf 1}$ and in addition  $\omega^{\xi}|_{Q_v}=\omega^{\mathbf 1}|_{Q_v}$.

   The coupling of these two measures follows the same method of randomized algorithm as in \cite{DCGR20}. In what follows, the definition of $A_t$, $B_t$ and $C_t$ are exactly the same as in \cite[Page 902-903]{DCGR20}: $A_t$ is the set of blocks that have been sampled up to time $t$, $B_t$ is the set of blocks that need to be checked, and $C_t$ is the set of edges for which $(\omega_e^{\xi},\omega_e^{\mathbf 1})$ has been sampled before $t$. 
    The whole algorithm is exactly the same as in \cite{DCGR20} except that the initial condition for $B_0$ needs a slight adjustment due to the change of region we are considering (i.e., we consider $W$ instead of $\Lambda_n$): we set
    $$B_0 =\{Q_v\in \mcc Q_l(\lamn): Q_v\cap \partial_i W \neq \emptyset \}\,.$$
In one sentence, this randomized sampling algorithm explores clusters of $B_0$ of blocks which are not very good (one may see \cite{DCGR20} for a formal definition).

 Following \cite{DCGR20} we can similarly derive that for a constant $c_3 = c_3(T, l) > 0$,
        \begin{equation}\label{eq:algorithm-goes-into-O}
          \sup_{A\subset\{0,1\}^{E(\mathsf Q_{\mathsf P^\mathrm o})}}|\phi_{p,W}^{\xi}(A)-\phi_{p, W}^{\mathbf 1}(A)|\leq \phi_{p,W}^{\xi, \mathbf 1}(C_T\cap \mathsf Q_{\mathsf P^{\mathrm o}}\neq \emptyset )\leq \exp(-c_3 \mathsf q)\,.
        \end{equation}
      The following is the proof for \eqref{eq:algorithm-goes-into-O} as argued in  \cite{DCGR20}: when $C_T\cap \mathsf Q_{\mathsf P^{\mathrm o}} \neq \emptyset$, there must exist a sequence of disjoint and not very good blocks $Q_{v_1},\cdots, Q_{v_s}$ (used by the former algorithm at times $t_1,\cdots, t_s$) such that the two blocks which are neighboring in the sequence has   $\ell_\infty$-distance at most $3l$. Therefore, \eqref{eq:algorithm-goes-into-O} follows from a union bound thanks to \eqref{eq-domination} and \cite[Proposition 1.5]{DCGR20}. This gives the weak mixing property as well as the property of exponentially
bounded controlling regions. By \cite[Theorem 3.3]{Alexander98}, we have for $c_4 = c_4(l, T)>0$ and for every $j$,
    $$|\fk+(\omega|_{\mathsf P_{j}^{\mathrm o}}=\omega^*|_{\mathsf P_{j}^{\mathrm o}}\mid \xi_j)-\fk+(\omega|_{\mathsf P_{j}^{\mathrm o}}=\omega^*|_{\mathsf P_{j}^{\mathrm o}} \mid \tilde \xi_j)|\leq e^{-c_4 \mathsf q} \fk+(\omega|_{\mathsf P_{j}^{\mathrm o}}=\omega^*|_{\mathsf P_{j}^{\mathrm o}}\mid \xi_j)\,.$$
    Combined with  \eqref{eq:conditional-prob-two-parts} and \eqref{eq:boundary-edges-estimate} then gives for $c_5 = c_5(l, T)>0$: 
    \begin{equation}
 \frac{\fk+(\omega|_{\mathsf{B\cup B'}}=\omega^*|_{\mathsf{B\cup B'}}\Big{|} \omega_\mathtt E=\omega^\mathtt E)}{\fk+(\omega|_{\mathsf{B\cup B'}}= \chi (\omega^*)|_{\mathsf{B\cup B'}}\Big{|} \omega_\mathtt E=\tilde\omega^\mathtt E)}\leq \exp(\frac{c_5\ln \mathsf q}{\mathsf q} L + c_5 L e^{-c_4 \mathsf q})\,.
    \end{equation}
    Taking $\mathsf q$ large enough then yields \eqref{eq-spatial-mixing-to-show} as required.

	\section{Localizing influence from disorder}\label{sec:localization-disorder}
	
	The main goal of this section is to prove Proposition~\ref{prop:inside-M-event-h-and-no-h} and Lemma~\ref{lem:flip-external-field-estimation}, for which the underlying intuition is that the influence from disorder mainly comes from disorder in a thin shell around $\mathsf B\cup \mathsf B'$. While this may be intuitive, the proof is not obvious. The major conceptual challenge is to show that the influence from disorder far away from $\mathsf B \cup \mathsf B'$ will not propagate via the interaction of FK configurations. Our definition of good boxes as well as the choice of $(\mathsf B, \mathsf B')$ is exactly for this purpose. We first prove Proposition~\ref{prop:inside-M-event-h-and-no-h}  in Section~\ref{sec:proof-prop-inside-M-event-h-and-no-h}, and then in Section~\ref{sec:proof-lem-flip-external-field-estimation} we prove Lemma~\ref{lem:flip-external-field-estimation} by a similar argument.

\subsection{Proof of Proposition~\ref{prop:inside-M-event-h-and-no-h}}\label{sec:proof-prop-inside-M-event-h-and-no-h}
Recall that $\mathtt E$ denotes the edges within
$\mathsf Q_{\Lambda_{\mathsf N}\setminus (\mathsf B \cup \mathsf B' \cup \mathsf S_1)}$ and recall \eqref{eq-def-Omega-B-B-prime-E}, \eqref{eq-def-Omega-omega-E} and \eqref{eq-def-Omega-Aux} for the definition of $\Omega_{\mathsf B, \mathsf B'}^{ \mathtt E}$, $\Omega(\omega^{\mathtt E})$ and $\Omega_{\mathsf {Aux}}(\omega^\mathtt E)$. We have the following expressions:
	\begin{align}
		&\fkf+{\tau_\omega(h)}(\Omega_{\mathrm{Aux}}(\omega^{\mathtt E})\mid\omega|_{\mathtt E}= \omega^\mathtt E)
		=
		\frac{\sum_{\omega\in \Omega_{\mathrm{Aux}}(\omega^{\mathtt E})}\fkf+{\tau_{\omega}(h)} {(\omega)}}
		{\sum_{\omega\in \Omega(\omega^{\mathtt E})}\fkf+{\tau_{\omega}(h)} (\omega)}\,,\label{eq:fraction-with-h}
	\end{align}
	which also holds for the version of $\fkf+{h}$ as well as for the special case of $\epsilon = 0$.
	Recall \eqref{eq-def-random-cluster} and \eqref{eq-def-FK-external-field} and note that the partition function (i.e., the normalizing factors in \eqref{eq-def-random-cluster} and \eqref{eq-def-FK-external-field}) get cancelled in the fraction in \eqref{eq:fraction-with-h}. 
    For $\omega^{\mathtt E} \in \Omega^{\mathtt E}_{\mathsf B, \mathsf B'}$, let $\mathfrak C^{\mathtt E} = \mathfrak C^{\mathtt E}(\omega^{\mathtt E})$ be the collection of $\omega^{\mathtt E}$-clusters on $\mathtt E$, and let $\mathcal C_*^{\mathtt E} \in \mathfrak C^{\mathtt E}$ be the cluster that is connected to $\pare \Lambda_N$, and recall $C^\mathtt E_\diamond$ (note that $C^{\mathtt E}_\diamond$ may or may not equal to $\mathcal C_*^{\mathtt E}$). Define
$$\mathrm{Ratio}(\omega^{\mathtt E}, h) = \exp(\epsilon(\tau_{\omega^\mathtt E}(h))_{\mathcal C_*^{\mathtt E}}/T)\prod_{\mathcal C^\mathtt E\in \mathfrak C^{\mathtt E}\setminus \{\mathcal C_*^{\mathtt E}\}} \cosh (\epsilon(\tau_{\omega^\mathtt E}(h))_{\mathcal C^\mathtt E}/T) \,,$$
which is well-defined since $\tau_{\omega}$ is a function of $\omega^\mathtt E$ (thus we write the heavier notation of $\tau_{\omega^\mathtt E}$ in this section to emphasize this point). 
For $\omega\in \Omega_{\mathrm{Aux}}(\omega^{\mathtt E})$, let $\mathfrak C = \mathfrak C(\omega)$ be the collection of $\omega$-clusters. In addition, let $\mathcal C_*$ be the $\omega$-cluster of $\pare \Lambda_N$, and let $\mcc C_\diamond \supset C^{\mathtt E}_\diamond$ be the $\omega$-cluster of $C^{\mathtt E}_\diamond$. Define
	$$\mathrm{Ratio}(\omega, h) = \frac{\parf {\tau_{\omega^\mathtt E}(h)} \fkf+{\tau_{\omega^\mathtt E}(h)} {(\omega)}}{\part \fk+{(\omega)}} = \exp(\epsilon(\tau_{\omega^\mathtt E}(h))_{\mathcal C_*}/T) \prod_{\mathcal C\in \mathfrak C \setminus \{\mathcal C_*\}} \cosh(\epsilon(\tau_{\omega^\mathtt E}(h))_{\mathcal C}/T)\,.$$
	The key challenge in comparing the above two ratios lies in controlling the cosine hyperbolic terms that cannot be cancelled. In order to address this, recall Definition~\ref{def-B-B-prime-U} and \eqref{eq-def-C-from-U}, and write $\mathtt C = \cup_{i=1}^n C_i$.
 Let $\Psi: \mathfrak C^{\mathtt E} \mapsto \mathfrak C$ so that $\mathcal C^\mathtt E \subset \Psi(\mathcal C^\mathtt E)$ for all $\mathcal C^\mathtt E\in \mathfrak C^{\mathtt E}$. For $\mathcal C\in \mathfrak C$, let $\cup \Psi^{-1}(\mathcal C) = \cup_{\mathcal C^\mathtt E \in \Psi^{-1}(\mathcal C)} \mathcal C^{\mathtt E}$ and let $\mathrm{Diff}(\mathcal C) = \mathcal C \setminus \cup\Psi^{-1}(\mathcal C)$ be the difference between $\mathcal C$ and $\cup\Psi^{-1}(\mathcal C)$.

For notation convenience, write $\mathrm{Shell} = \mathsf Q_{\mathrm{Ball}(\mathsf B \cup \mathsf B', 4\mathsf k)}$. 
The following are some basic facts to be used in our proof.

First, for $\omega^\mathtt E\in \Omega^{\mathtt E}_{\mathsf B, \mathsf B'}$, by \eqref{eq-red-surounding} and \eqref{eq-R-to-Lambda-N} (as well as the definition of good box) we see that for all $\omega \in \Omega_{\mathrm{Aux}}(\omega^\mathtt E)$
\begin{equation}\label{eq-structure-C-*}
\mathrm{Diff}(\mathcal C_\diamond)\setminus \mathtt C \subset \mathrm{Shell}\,.
\end{equation}
In the case that $\mcc C^{\mathtt E}_*\cap  C^{\mathtt E}_\diamond=\emptyset$, recalling the fact that $\partial_{\mathsf k}\mathsf U_*$ is a connected set with all points good, we see that for all $\omega\in \Omega(\omega^{\mathtt E})$,
\begin{equation}\label{eq:structure-C-*-in-U-*}
    \mcc C_*\subset  \psi(\partial_{\mathsf k}\mathsf U_*)^c\ ,\quad  \mcc C_*=\mcc C_*^{\mathtt E} \quad \mbox{ and  }\quad \mcc C_*\cap \mcc C_\diamond=\emptyset\,.
\end{equation}
This is because otherwise $\mcc C_*$ passes through $\mathsf Q_{\partial_\mathsf k \mathsf U_*}$ and thus is connected to $C^\mathtt E_\diamond$. In the other case that $\mcc C^{\mathtt E}_*=  C^{\mathtt E}_\diamond$ (that is, $C^\mathtt E_\diamond$ is connected to $\pare \lamn$), by \eqref{eq-structure-C-*} we have $\mathrm{Diff}(\mathcal C_*)\setminus \mathtt C \subset \mathrm{Shell}$.

Next, we claim that for $\mathcal C\in \mathfrak C$ with $\mathcal C\neq \mathcal C_\diamond$,
\begin{equation}\label{eq:structure-of-C}
    \mathrm{Diff}(\mathcal C)\subset \mathrm{Shell}.
\end{equation}
For the purpose of verifying \eqref{eq:structure-of-C} and for purpose of analysis later, we divide $\mathcal C\in \mathfrak C$ with $\mathcal C\neq \mathcal C_\diamond$ into the following three cases:
\begin{enumerate}[(i)]
    \item $\mcc C\subset U_i$ for some $i=1,\cdots, n,*$ or $\mcc C\subset (\cup U_i\cup U_*)^c$;
    \item $\mcc C\cap \mathtt C\neq \emptyset$ and $\mcc C$ is not in Case (i);
    \item $\mcc C \cap \mathtt C =  \emptyset$ and $\mcc C$ is not in Case (i).
\end{enumerate}
We now verify \eqref{eq:structure-of-C} in Case (i). If the former occurs we have $\mathrm{Diff}(\mathcal C)=\emptyset$; if the latter occurs we have $\mathrm{Diff}(\mathcal C)=\mathcal C\subset \mathsf{Shell}$. The verification of \eqref{eq:structure-of-C} in Cases (ii) and (iii) follow from the same argument (we separate these two cases for later analysis):  By the assumption that this is not in Case (i), we see that $\mathcal C\setminus \mathtt C$ can be decomposed into components each of which intersects with $\mathsf Q_{\partial_{\mathsf k}\mathsf U_j}$ for some $j=1, \ldots, n, *$. By \eqref{eq-def-C-from-U}, all these components have diameters at most $\mathsf q$, yielding \eqref{eq:structure-of-C}.

The following equality is obvious:
\begin{equation}\label{eq:sum-of-inverse-image}
    \sum_{\mcc C\in \mathfrak C}|\Psi^{-1}(\mcc C)\cap \{C_1, \ldots, C_n\}|=n.
\end{equation}

For  $\omega^{\mathtt E}$ with $C^{\mathtt E}_\diamond \leftrightarrow \pare \lamn$, we have $\mathcal C^{\mathtt E}_*= C_\diamond^{\mathtt E}$ and $\mcc C_*=\mcc C_\diamond$. By \eqref{eq-def-flipping}, we have $(\tau_{\omega^\mathtt E}(h))_{C_i} \geq 0$ for $i=1, \ldots, n$. Combined with \eqref{eq-structure-C-*} and \eqref{eq:sum-of-inverse-image}, it yields  that (in what follows we write $I_* = \{1\leq i\leq n: C_i \subset \mathcal C_*\}$ and thus $|I_*| = |\Psi^{-1}(\mcc C_*) \cap \{C_1, \ldots, C_n\}|$)
    \begin{align}
    &\big|\frac{\epsilon(\tau_{\omega^\mathtt E}(h))_{\mathcal C_*}}{T} - \frac{\epsilon(\tau_{\omega^\mathtt E}(h))_{\mathcal C^{\mathtt E}_*}}{T} - \sum_{\mathcal C^{\mathtt E} \in \Phi^{-1}(\mathcal C_*)\setminus \{\mathcal C_*^{\mathtt E}\}} \log  \cosh \frac{\epsilon(\tau_{\omega^\mathtt E}(h)_{\mathcal C^\mathtt E}}{T}\big| \nonumber\\
    \leq& \sum_{i\in I_*}|\frac{\epsilon(\tau_{\omega^\mathtt E}(h))_{C_i}}{T} - \log  \cosh \frac{\epsilon(\tau_{\omega^\mathtt E}(h)_{C_i}}{T}| + c_6\epsilon\sum_{v\in \mathrm{Shell} \cap \mathcal C_*} |h_v| \nonumber\\
    \leq & c_7  |I_*|+ c_6\epsilon\sum_{v\in \mathrm{Shell} \cap \mathcal C_*} |h_v|\,, \label{eq-control-C-*-Diff}
    \end{align}
    where $c_6 =c_6(T)>0$ and $c_7 = c_7(T)>0$ are constants. We now explain the first inequality in \eqref{eq-control-C-*-Diff}. Set $\hat h$ such that $\hat h_v = 0$ for all $v\in \mathrm{Shell \cap \mathcal C_*}$ and $\hat h_v = h_v$ otherwise. Then we can replace $h$ by $\hat h$ in the target difference for an approximation, and by \eqref{eq-structure-C-*} this approximation has an error upper-bounded by $c_6\epsilon\sum_{v\in \mathrm{Shell} \cap \mathcal C_*} |h_v|$. In addition, the target difference with $h$ replaced by $\hat h$ is upper-bounded by $ \sum_{i\in I_*}|\frac{\eps(\tau_{\omega^\mathtt E}(h))_{C_i}}{T} - \log  \cosh \frac{\epsilon(\tau_{\omega^\mathtt E}(h)_{C_i}}{T}| \leq c_7 |I_*|$. By the triangle inequality, this verifies the first inequality in \eqref{eq-control-C-*-Diff}.

For  $\omega^{\mathtt E}$ with $C^{\mathtt E}_\diamond \nleftrightarrow \pare \lamn$, we have $\mathcal C^{\mathtt E}_*\cap \mcc C_\diamond^{\mathtt E}=\emptyset=\mcc C_* \cap \mcc C_\diamond$. Since in this case $\mathcal C_* = \mathcal C^\mathtt E_*$, we have
    \begin{equation}\label{eq:control-c*-diff-2}
        \big|\frac{\epsilon(\tau_{\omega^\mathtt E}(h))_{\mathcal C_*}}{T} - \frac{\epsilon(\tau_{\omega^\mathtt E}(h))_{\mathcal C^{\mathtt E}_*}}{T}\big|=0\,.
    \end{equation}
In addition,  by \eqref{eq-def-flipping} we have for $\omega^{\mathtt E}$ with $C^{\mathtt E}_\diamond \nleftrightarrow \pare \lamn$
\begin{equation}\label{eq-signs-disorder-C-i}
\mbox{$(\tau_{\omega^{\mathtt E}}(h))_{C_i}$'s have the same sign with $h_{C^\mathtt E_\diamond}$ for all $i = 1, 
\ldots, n$.}
\end{equation}
Therefore, by a similar derivation to \eqref{eq-control-C-*-Diff} we have (below we may increase the value of $c_6, c_7$ if necessary)
\begin{align}\label{eq:control-c-e-diff}
    &\big|\log  \cosh \frac{\epsilon(\tau_{\omega^\mathtt E}(h))_{\mathcal C_\diamond}}{T} -\sum_{\mathcal C^{\mathtt E} \in \Phi^{-1}(\mathcal C_\diamond)} \log \cosh \frac{\epsilon(\tau_{\omega^\mathtt E}(h)_{\mathcal C^\mathtt E}}{T}\big|  \nonumber \\
    &\leq c_7|\{C_1, \ldots, C_n\} \cap \Psi^{-1}(\mathcal C_\diamond)|+c_6\epsilon \sum_{v\in \mathrm{Shell} \cap \mathcal C_\diamond} |h_v|.
    \end{align}
 Furthermore, for $\mathcal C\in \mathfrak C \setminus \{\mathcal C_\diamond, \mathcal C_*\}$, we claim that
    \begin{align}\label{eq-control-C-Diff}
   & |\log \cosh(\epsilon(\tau_{\omega^\mathtt E}(h)_{\mathcal C}/T) - \sum_{\mathcal C^{\mathtt E}\in \Phi^{-1}(\mathcal C)}\log \cosh (\epsilon(\tau_{\omega^\mathtt E}(h)_{\mathcal C^\mathtt E}/T) | \nonumber\\ 
    & \leq c_7 |\Psi^{-1}(\mathcal C)\cap \{C_1, \ldots, C_n\}| + c_6 \epsilon \sum_{v\in \mathrm{Shell} \cap \mathcal C} |h_v|\,.
    \end{align} 
 In order to see \eqref{eq-control-C-Diff}, we recall the three cases in verifying \eqref{eq:structure-of-C}. In Case (i), the above difference is 0.
    In Case (ii), we use a derivation as for \eqref{eq-control-C-*-Diff}. We have
    \begin{align*}
    &\big |\log \cosh(\epsilon(\tau_{\omega^\mathtt E}(h))_{\mathcal C}/T) - |\epsilon(\tau_{\omega^\mathtt E}(h))_{\mathcal C}/T| \big|\\
    & + \sum_{\mathcal C^{\mathtt E}\in \Phi^{-1}(\mathcal C)} \big|\log\cosh(\epsilon(\tau_{\omega^\mathtt E}(h))_{\mathcal C^\mathtt E}/T) - |\epsilon(\tau_{\omega^\mathtt E}(h))_{\mathcal C^\mathtt E})/T| \big|\end{align*} is upper-bounded by $\frac{c_6}{2} \epsilon \sum_{v\in \mathrm{Shell} \cap \mathcal C} |h_v|+ \frac{c_7}{2}|\Phi^{-1}(\mathcal C)\cap \{C_1, \ldots, C_n\}|$. In addition, by \eqref{eq-signs-disorder-C-i} we have $$\big||\epsilon(\tau_{\omega^\mathtt E}(h))_{\mathcal C}/T| - \sum_{\mathcal C^{\mathtt E}\in \Phi^{-1}(\mathcal C)}|\epsilon(\tau_{\omega^\mathtt E}(h))_{\mathcal C^\mathtt E})/T|\big| \leq\frac{c_6}{2} \epsilon \sum_{v\in \mathrm{Shell} \cap \mathcal C} |h_v|\,.$$ Altogether, gives \eqref{eq-control-C-Diff}.
   In Case (iii), the inequality holds since $\mathcal C \subset \mathrm{Shell}$ by \eqref{eq:structure-of-C}. This completes the verification of \eqref{eq-control-C-Diff}.

    Combining \eqref{eq-control-C-*-Diff}, \eqref{eq:control-c*-diff-2}, \eqref{eq:control-c-e-diff}, \eqref{eq:sum-of-inverse-image}and \eqref{eq-control-C-Diff}, we see that
    $$|\log \mathrm{Ratio}(\omega^{\mathtt E}, h) - \log \mathrm{Ratio}(\omega, h)| \leq c_7 n + c_6\epsilon \sum_{v\in \mathrm{Shell}} |h_v|\,.$$
    Combined with Lemma~\ref{lem:comparez}, \eqref{eq:number-of-U} and \eqref{eq:fraction-with-h}, this completes the proof of Proposition~\ref{prop:inside-M-event-h-and-no-h}.
	
\subsection{Proof of Lemma~\ref{lem:flip-external-field-estimation}}\label{sec:proof-lem-flip-external-field-estimation}

We continue to use notations from the previous subsection. We have that for $\omega^{\mathtt E} \in \Omega^{\mathtt E}_{\mathsf B, \mathsf B'}$ and $\omega\in \Omega_{\mathrm{Aux}}(\omega^{\mathtt E})$
\begin{align}
&\log (\vfkf+{h}(\omega)\parz) -    \log(\vfkf+{\tau_{\omega}(h)}(\omega) \parf {\tau_{\omega(h)}}) \nonumber\\
& = \epsilon \frac{h_{\mathcal C_*} - (\tau_{\omega^\mathtt E}(h))_{\mathcal C_*}}{T} + \sum_{\mathcal C\in \mathfrak C\setminus \{\mathcal C_*\}} (\log \cosh (\epsilon h_{\mathcal C}/T) - \log \cosh (\epsilon \tau_{\omega^\mathtt E}(h)_{\mathcal C}/T)\,. \label{eq-log-difference-flip-disorder}
\end{align}

By \eqref{eq-def-flipping}, when $\mcc C_*\neq \mcc C_\diamond$ (so $\mcc C^\mathtt E_* \neq C^\mathtt E_\diamond$), 
\begin{equation}\label{eq-contro-C-*-Diff-lemma-0}
    h_{\mathcal C_*} - (\tau_{\omega^\mathtt E}(h))_{\mathcal C_*}=0\,.
\end{equation}
In addition, when $\mcc C_*=\mcc C_\diamond$, by \eqref{eq-def-flipping} we have  $h_{\mcc C_*\cap C_i}-((\tau_{\omega^\mathtt E}(h))_{\mathcal C_*\cap C_i}\leq 0$, for each $i=1,\cdots, n$. Therefore, by \eqref{eq-structure-C-*} we have
\begin{equation}\label{eq-control-C-*-Diff-lemma}
    h_{\mathcal C_*} - (\tau_{\omega^\mathtt E}(h))_{\mathcal C_*}\leq \sum_{v\in \mathrm{Shell} \cap \mathcal C_*} |h_v|\,.
\end{equation}
For $\mathcal C \neq \mathcal C_*$, we claim that
\begin{align}\label{eq-control-C-Diff-lemma}
\log \cosh (\eps h_{\mathcal C}/T) - \log \cosh (\eps\tau_{\omega^\mathtt E}(h)_{\mathcal C}/T)
 \leq c_8\epsilon \sum_{v\in \mathrm{Shell} \cap \mathcal C} |h_v|\,,
\end{align}
where $c_8 = c_8(T) >0$ is a constant. We now verify \eqref{eq-control-C-Diff-lemma} and we recall the three cases in verifying \eqref{eq:structure-of-C}.  In Case (i) the above difference is 0. In Case (ii),  by \eqref{eq-signs-disorder-C-i} we can deduce \eqref{eq-control-C-Diff-lemma} from \eqref{eq-structure-C-*} and \eqref{eq:structure-of-C}.
In Case (iii), we deduce \eqref{eq-control-C-Diff-lemma} from \eqref{eq-structure-C-*} and \eqref{eq:structure-of-C}. This completes the verification of \eqref{eq-control-C-Diff-lemma}. Plugging \eqref{eq-control-C-*-Diff-lemma} and \eqref{eq-control-C-Diff-lemma} into \eqref{eq-log-difference-flip-disorder}, we complete the proof of the lemma.

\medskip

\noindent {\bf Acknowledgement.} We warmly thank Subhajit Goswami, Jianping Jiang, Jian Song, Rongfeng Sun and Zijie Zhuang for helpful discussions.

\end{document}